\documentclass{article}
\usepackage{graphicx}
\usepackage{url}
\usepackage{amsmath}
\usepackage{amsfonts}
\usepackage{verbatim}
\usepackage{amssymb}
\usepackage{amsthm}
\usepackage{color}
\usepackage{epstopdf}

\newtheorem{thm}{Theorem}
\newtheorem{defn}[thm]{Definition}
\newtheorem{prop}[thm]{Proposition}
\newtheorem{cor}[thm]{Corollary}
\newtheorem{lem}[thm]{Lemma}

\begin{document}

\title{Invariant Measures for Hybrid Stochastic Systems}

\author{Xavier Garcia\thanks{
Department of Mathematics, University of Minnesota - Twin Cities, Minneapolis, MN 55455, USA
(garci363@umn.edu). Research supported by DMS 0750986 and DMS 0502354} \and Jennifer Kunze\thanks{
Department of Mathematics, Saint Mary's College of Maryland, St Marys City, MD 20686, USA
($\mbox{jckunze@smcm.edu}$). Research supported by DMS 0750986 } \and Thomas Rudelius\thanks{
Department of Mathematics, Cornell University, Ithaca, NY 14850, USA
(twr27@cornell.edu). Research supported by DMS 0750986 } \and  Anthony Sanchez\thanks{
Department of Mathematics, Arizona State University, Tempe, AZ 85281, USA
($\mbox{Anthony.Sanchez.1@asu.edu}$). Research supported by DMS 0750986 and DMS 0502354 } \and  Sijing Shao\thanks{
Department of Mathematics, Iowa State University, Ames, IA 50014, USA
(sshao@iastate.edu). Research supported by Iowa State University } \and Emily Speranza\thanks{
Department of Mathematics, Carroll College, Helena, MT 59625, USA
($\mbox{esperanza@carroll.edu}$). Research supported by DMS 0750986 }
\and Chad Vidden\thanks{
Department of Mathematics, Iowa State University, Ames, IA 50014, USA
($\mbox{cvidden@iastate.edu}$).  Research supported by Iowa State University.}}
\date{\today}

\maketitle


\abstract
In this paper, we seek to understand the behavior of dynamical systems
that are perturbed by a parameter that changes discretely in time.
If we impose certain conditions, we can study certain embedded systems
within a hybrid system as time-homogeneous Markov processes. In particular,
we prove the existence of invariant measures for each embedded
system and relate the invariant measures for the various systems through
the flow. We calculate these invariant measures explicitly in several illustrative
examples.


\section{Introduction}

An understanding of dynamical systems allows one to analyze the way processes evolve through time. Usually, such systems are given by differential equations that model real world phenomena. Unfortunately, these models are limited in that they cannot account for random events that may occur in application. These stochastic developments, however, may sometimes be modeled with Markov processes, and in particular with Markov chains. We can unite the two models in order to see how these dynamical systems behave with the perturbation induced by the Markov processes, creating a hybrid system consisting of the two components. Complicating matters, these hybrid systems can be described in either continuous or discrete time.

The focus of this paper is studying the way these hybrid systems behave as they evolve. We begin by defining limit sets for a dynamical system and stochastic processes. We next examine the limit sets of these hybrid systems and what happens as they approach the limit sets. Concurrently, we define invariant measures and prove their existence for hybrid systems while relating these measures to the flow. In addition, we supply examples with visuals that provide insight to the behavior of hybrid systems.


\section{The Stochastic Hybrid System}

In this section, we define a hybrid system.

\begin{defn}\label{timehomogeneous}
A Markov process $X_t$ is called time-homogeneous on $T$ if, for all $t_1, t_2, k \in T$ and for any sets $A_1,A_2\in S$,
$$P(X_{t_1+k} \in A_1 | X_{t_1} \in A_2) = P(X_{t_2+k} \in A_1 | X_{t_2} \in A_2) .$$ 
Otherwise, it is called time-inhomogeneous.
\end{defn}

\begin{defn}\label{markovchain2}
A Markov chain $X_n$ is a Markov process for which perturbations occur on a discrete time set $T$ and finite state space $S$.
\end{defn}

\noindent
For a Markov chain on the finite state space $S$ with cardinality $|S|$, it is useful to describe the probabilities of transitioning from one state to another with a transition matrix 
$$
Q \equiv \left ( \begin{array}{ccccc}
P_{1 \rightarrow 1} & .&.&.&P_{1 \rightarrow |S|} \\
.&&&&. \\
.&&&&. \\
.&&&&. \\
P_{|S| \rightarrow 1} & . & . & . & P_{|S| \rightarrow |S|}\\
\end{array} \right )
$$
where $P_{i \rightarrow j}$ is the probability of transitioning from state $s_i\in S$ to state $s_j\in S$.

Also, for the purposes of this paper, we suppose that our Markov chain transitions occur regularly at times $t=nh$ for some length of time $h\in\mathbb{R}^+$ and for all $n\in\mathbb{N}$.

\begin{defn}\label{markovchain}
Let $\{X_n\}$, for $X_n \in S$ and $ n \in \mathbb{N}$, be a sequence of states determined by a Markov chain.

For $t \in \mathbb{R}^+$, define the Markov chain perturbation $Z_t=X_{\left \lfloor\frac{t}{h} \right \rfloor}$, where $\left \lfloor \frac{t}{h} \right \rfloor$ is the greatest integer less than or equal to $\frac{t}{h}$.
\end{defn}

Note that $Z_t$, instead of being defined only on discrete time values like a Markov chain, is instead a stepwise function defined on continuous time.

\begin{defn}\label{varphi}
Given a metric space $M$ and state space $S$ as above, define a dynamical system $\varphi$ with random perturbation function $Z_t$, as given in Definition \ref{markovchain}, by
$$ \varphi: \mathbb{R}^+ \times M \times S \rightarrow M$$
with 
$$\varphi(t,x_0,Z_0) = \varphi_{Z_t}(t-nh, \varphi_{Z_{nh}}(h,  ... \varphi_{Z_{2h}}(h, \varphi_{Z_h}(h,\varphi_{Z_0}(h,x_0))))) $$
where $\varphi_{Z_k}$ represents the deterministic dynamical system $\varphi$ evaluated in state $Z_k$ and $nh$ is the largest multiple of $h$ less than $t$.
\end{defn}

\noindent
For ease of notation, let 
$$x_t=\varphi(t, x_0, Z_0)\in M$$ 
represent the position of the system at time $t$.

\begin{defn}\label{hybridsys}
Let 
$$ Y_t = \left ( \begin{array} {c}
x_t \\
Z_t
\end{array} \right )$$
define the hybrid system at time $t$.  In other words, the hybrid system consists of both a position $x_t=\varphi(t, x_0, Z_0) \in M$ and a state $Z_t \in S$.
\end{defn}

The $\omega$-limit set has the following generalization in a hybrid system.

\begin{defn}\label{C}
The stochastic limit set $C(x)$ for an element of our state space $x\in M$ and the hybrid system given above is the subset of $M$ with the following three properties:
\begin{enumerate}
\item Given $y\in M$ and $t_k\rightarrow\infty$ such that $x_{t_k}\rightarrow y$, $P(y\in C(x))=1$.
\item $C(x)$ is closed.
\item $C(x)$ is minimal: if some set $C'(x)$ has properties 1 and 2, then $C \subseteq C'$.
\end{enumerate}
\end{defn}


\section{The Hybrid System as a Markov Process}

\begin{lem}
Each of the following is a Markov process:
\\(i) any deterministic dynamical system $\varphi (t,x_0)$, as in Definition \ref{dynsys}.
\\(ii) any Markov chain perturbation $Z_t$, as in Definition \ref{markovchain}.
\\(iii) the corresponding hybrid system $Y_t$, as in Definition \ref{hybridsys}.
\end{lem}

\begin{proof}  (i) Any deterministic system is trivially a Markov process, since $\varphi(t,x_0)$ is uniquely determined by $\varphi(\tau,x_0)$ at any single past time $\tau\in\mathbb{R}^+$.

(ii) By definition, a Markov chain is a Markov process.  However, the Markov chain perturbation $Z_t$ is not exactly a Markov chain.  A Markov chain exists on a discrete time set, in our case given by $T=\{t\in\mathbb{R}^+ | t=nh$ for some $n\in\mathbb{N}\}$; conversely, the time set of $Z_t$ is $\mathbb{R}^+$, with transitions between states ocurring on the previous time set (that is, at $t \equiv 0$ mod $h$).  Despite this difference, $Z_t$ maintains the Markov property: we can compute $P(Z_t \in A)$ for any set $A$ based solely on $Z_{\tau_1}$ and the values of the times $t$ and $\tau_1$.  Explicitly, the probability that $Z_t$ will be in state $s_i$ at time $t$ is given by 
$$P(Z_t = s_i) = \left( (Q^T)^n \right)_{ij}$$
where $n$ is the number of integer multiples of $h$ (i.e. the number of transitions that occur) between $t$ and $\tau_1$.  Clearly, this is independent of the states $Z_{\tau_i}$ for $i > 1$, so that the random perturbation is indeed a Markov process.

(iii)  Now, keeping in mind that the hybrid system $Y_t$ consists of both a location $x_t\in M$ in the state space and a value $Z_t\in S$ of the random component, we can combine (i) and (ii) to see that the entire system is also a Markov process.  We see from (ii) that $Z_t$ follows a Markov process.  Furthermore, $P(x_t \in A_x)$ at time $t$ depends solely on the location $x_{\tau_1}$ at any time $\tau_1<t$ and the states of the random perturbation sequence $Z$ between $t$ and $\tau_1$, regardless of any past behavior of the system. Hence, for any collection of sets $A_{\alpha}$, $\alpha\in\mathbb{N}$,
$$P(Z_t \in A_z | Z_{\tau_1} \in A_{z_1}, Z_{\tau_2} \in A_{z_2},...,Z_{\tau_n} \in A_{z_n}) = P(Z_t \in A_z | Z_{\tau_1} \in A_{z_1})$$
and
$$P(x_t \in A_x | x_{\tau_1} \in A_{x_1}, x_{\tau_2} \in A_{x_2},...,x_{\tau_n} \in A_{x_n}) = P(x_t \in A_x | x_{\tau_1} \in A_{x_1}).$$
So,
$$P(Y_t \in A_y | Y_{\tau_1} \in A_{y_1}, Y_{\tau_2} \in A_{y_2},...,Y_{\tau_n} \in A_{y_n}) = P(Y_t \in A_y | Y_{\tau_1} \in A_{y_1}).$$
Thus, the hybrid system is a Markov process.  
\end{proof}

Unfortunately, the hybrid system is not time-homogeneous.  Recall that state transitions of $Z_t$ occur at times $t=nh$ for $n\in\mathbb{N}$.  So, the state of the system at time $\frac{h}{4}$ uniquely determines the system at $\frac{3h}{4}$, since there is no transition in this interval.  However, the system at time $\frac{5h}{4}$ is not determined uniquely by the system at $\frac{3h}{4}$, since a stochastic transition occurs at $t = h \in [\frac{3}{4},\frac{5}{4}]$.  Therefore, with $t_1 = \frac{h}{4}$, $t_2 = \frac{3h}{4}$, and $k = \frac{1}{2}$, 
$$ P(Y_{\frac{h}{4}+ \frac{1}{2}}\in A | Y_\frac{h}{4} \in A_0) \neq P(Y_{\frac{3h}{4}+\frac{1}{2}} \in A | Y_\frac{3h}{4} \in A_0) ,$$
violating Definition \ref{timehomogeneous}.  However, in order to satisfy the hypotheses of the Krylov-Bogolyubov theorem \cite{Hairer, Ergodic} found in Theorem \ref{Krylov-Bogolyubov}, the hybrid system must be time-homogeneous.  

To create a time-homogeneous system, we restrict the time set on which our Markov process is defined.  Instead of allowing our time set $\{t, \tau_1, \tau_2, \tau_3, ... , \tau_n\} \subset \mathbb{R}^+$ to be any decreasing sequence of real numbers, we create time sets $t_0+nh$ for each $t_0\in [0,h)$ and $n\in\mathbb{N}$.  In other words, we define a different time set for each value $t_0<h$ with
\begin{center}$ \{t\in\mathbb{R}^+ | t = t_0+nh$ for some $n\in\mathbb{N}\} $ . \end{center}
We call the hybrid system on these multiple, restricted time sets the discrete system.

\begin{prop}\label{discretetomarkov}
The discrete hybrid system above is a time-homogeneous Markov process.
\end{prop}

\begin{proof}  
First, we must show that the discrete hybrid system is a Markov process at all.  This follows immediately from the proof that our original hybrid system is a Markov process.  Since the Markov property (Definition \ref{markovprocess}) holds for all $t, \tau_1, \tau_2, ..., \tau_n \in \mathbb{R}^+$, it must necessarily hold for the specific time sets 
\begin{center} $\{t\in\mathbb{R}^+ | \exists n\in\mathbb{N}$ such that $t = t_0+nh\}$\end{center}
for each $t_0<h$.

Now, it remains to show that this system is time-homogeneous.  Recall that the time-continuous hybrid system failed to be time-homogeneous because its $Z_t$ component was not time-homogeneous.  Although transitions occurred only at regular, discrete time values, a test interval could be of any length; an interval of size $\frac{h}{2}$, for example, might contain either $0$ or $1$ transitions.  However, because our discrete system creates separate time sets, any time interval - starting and ending within the same time set - must be of length $nh$ for some $n\in\mathbb{N}$, and thus will contain precisely $n$ potential transitions.  So, taking $t_1,t_2\in\mathbb{R}^+$, we know
$$ P(Y_{t_1+nh} \in A | Y_{t_1} \in A_0) = P(Y_{t_2+nh} \in A | Y_{t_2} \in A_0) .$$
Note that the first component of the hybrid system, $x_t$, is also time-homogeneous under the discrete time system.  Given $Z_t$, it can be treated as a deterministic system, and therefore time-homogeneous.
Thus, the discrete hybrid system is time-homogeneous.
\end{proof}

\section{Invariant Measures for the Hybrid System}

We now introduce several definitions that will lead to the main results of this paper.

\begin{defn}\label{markovinvm}
Consider a hybrid system $Y_t$ and a $\sigma$-algebra $\Sigma$ on the space $M$.  A measure $\mu$ on $M$ is invariant if, for all sets $A\in\Sigma$ and all times $t\in\mathbb{R}^+$, 
$$\mu (A) = \int_{x_0 \in M} P(x_t\in A) \mu (dx) .$$
\end{defn}

\begin{defn}\label{tight}
 Let $(M, \mathcal{T})$ be a topological space, and let $\Sigma$ be a $\sigma$-algebra on $M$ that contains the topology $\mathcal{T}$. Let $\mathcal{M}$ be a collection of probability measures defined on $\Sigma$. The collection $\mathcal{M}$ is called tight if, for any $\epsilon > 0$, there is a compact subset $K_{\epsilon}$ of $M$ such that, for any measure  $\mu$ in $\mathcal{M}$, 
$$ \mu (M \backslash K_{\epsilon}) < \epsilon .$$
\end{defn}

\noindent
Note that, since $\mu$ is a probability measure, it is equivalent to say $\mu(K_{\epsilon}) > 1 - \epsilon$.

The following definitions are from \cite{Hairer}.

\begin{defn}\label{Prokhorov}
Let $(M,\rho)$ be a separable metric space. Let $\{\mathcal{P}(M)\}$ denote the collection of all probability measures defined on $M$ (with its Borel $\sigma$-algebra). A collection $K \subset \{\mathcal{P}(M)\}$ of probability measures is tight if and only if $K$ is sequentially compact in the space  equipped with the topology of weak convergence.
\end{defn}

\begin{defn}\label{markovoperator}
Consider $M$ with $\sigma$-algebra $\Sigma$.  Let $C^0(M,\mathbb{R})$ denote the set of continuous functions from $M$ to $\mathbb{R}$.  The probability measure $\mathcal{P}(t,x,\cdot)$ on $\Sigma$ induces a map
$$\mathcal{P}_t(x):C^0(M,\mathbb{R}) \rightarrow \mathbb{R}$$
with
$$\mathcal{P}_t(x)(f) = \int_{y \in M}{f(y)\mathcal{P}(t,x,dy)} .$$
$\mathcal{P}_t$ is called a Markov operator.
\end{defn}

\begin{defn}\label{Feller}
A Markov operator $\mathcal{P}$ is Feller if $\mathcal{P} \varphi$ is continuous for every continuous bounded function $\varphi: X \rightarrow R$. In other words, it is Feller if and only if the map $x \mapsto \mathcal{P}(x, \cdot)$ is continuous in the topology of weak convergence.
\end{defn}

We state the Krylov-Bogolyubov Theorem without proof.

\begin{thm}\label{Krylov-Bogolyubov}  
{\rm \textbf{(Krylov-Bogolyubov)}}
Let $\mathcal{P}$ be a Feller Markov operator over a complete and separable space X. Assume that there exists $x_0 \in X$ such that the sequence $\mathcal{P}^{n}(x_0, \cdot)$ is tight. Then, there exists at least one invariant probability measure for $\mathcal{P}$. 
\end{thm}

\noindent
We now show that the conditions of the theorem are satisfied by the discrete hybrid system, yielding the existence of invariant measures as a corollary.

\begin{lem}\label{pFeller}
Given $t_0 \in [0,h)$, the discrete hybrid system Markov operators $\mathcal{P}_n$ for $n \in \mathbb{N}$ given by
$$\mathcal{P}_nf(Y) \equiv \int_{M \times S}{f(Y_1)\mathcal{P}(nh,Y,dY_1)}$$
are Feller.
\end{lem}

\begin{proof}
We begin by showing that $\mathcal{P}_1$ is Feller.  By induction, it follows that $\mathcal{P}_n$ is Feller for all $n\in\mathbb{N}$.  It is clear that there are only finitely many possible outcomes of running the hybrid system for time $h$.  Namely, there are at most $|S|$ possible outcomes, where $|S|$ denotes the cardinality of $S$.  Given 
$$Y_0 = 
\left ( \begin{array} {c}
x_0\\
Z_0=s_i
\end{array} \right )\in M \times S ,$$
the only possible outcomes at time $t=1$ are
$$ Y_1^j = \left ( \begin{array} {c}
\varphi_j(t_0,\varphi_i(h-t_0, x)) \\
s_j
\end{array} \right )$$
for $j \in \{1,...,|S|\}$ where $\varphi_i, \varphi_j$ are the flows of the dynamical systems corresponding to states $s_i$ and $s_j$, respectively.  The probability of the $j^{th}$ outcome is given by
$ P_{i \rightarrow j}$, the probability of transitioning from state $s_i$ to state $s_j$.  Therefore,
$$\mathcal{P}_1f(Y) = \int_{M \times S}{f(Y_1)\mathcal{P}(h,Y,dY_1)} = \sum_{j = 1}^{|S|}{P_{i \rightarrow j}f(Y_1^j)} $$

Each $\varphi_i$ is continuous under the assumption that each flow is continuous with respect to its initial conditions.  The map from $s_i$ to $s_j$ is continuous since $S$ is finite, so every set is open and hence the inverse image of any open set is open.  The function $f$ is continuous by hypothesis, and any finite sum of continuous functions is also continuous.  Therefore $\mathcal{P}_1f$ is also continuous, and hence $\mathcal{P}_1$ is Feller.

\end{proof}

We see now that the conditions of the Krylov-Bogolyubov Theorem (\ref{Krylov-Bogolyubov}) hold. Namely, because $M$ and $S$ are compact (the former by assumption, the latter since it is finite), $M \times S$ is compact. Thus, any collection of measures is automatically tight, since we can take $K_{\epsilon} = X$. It is well-known that any compact metric space is also complete and separable.  Applying Theorem \ref{Krylov-Bogolyubov}, then, gives the following corollary, which is one of the primary results of the paper.

\begin{cor}\label{discretemeasure}
The discrete hybrid system has an invariant measure for each $t_0 \in [0, h)$.
\end{cor}

So, rather than speaking of an invariant measure for the time-continuous hybrid system, we can instead imagine a periodic invariant measure cycling continuously through $h$.  That is, for each time $t_0 \in [0,h)$, there exists a measure $\mu_{t_0}$ such that for $t \equiv 0$ (mod $h$), 
$$\mu_{t_0} (A) = \int_{Y \in M \times S} \mathcal{P}(t,Y,A) d\mu_{t_0} .$$
The measure $\mu_{t_0}$ above is a measure on the product space $M \times S$, since this is where the hybrid system lives.  However, what we are really after is an invariant measure on just $M$, the space where the dynamical system part of the hybrid system lives.  Fortunately, we can define a measure on $M$ by the following construction.

\begin{prop}\label{inducedmeasure}
Given $\mu_t$, an invariant probability measure on $M \times S$, the function 
$$
\tilde{\mu}_t(A) \equiv \mu_t(A,S)
$$
where $A \subseteq M$ is an invariant probability measure on $M$.
\end{prop}
\begin{proof}
The fact that $\tilde{\mu}_t$ is a probability measure follows almost immediately from the fact that $\mu_t$ is a probability measure.  The probability that $x_t \in \emptyset$ is 0, so $\tilde{\mu}_t(\emptyset) = 0$.  The probability that $x_t \in M$ is $1$, so $\tilde{\mu}_t(M) = 1$.  Countable additivity of $\tilde{\mu}_t$ follows from countable additivity of $\mu_t$.  Therefore, $\tilde{\mu}_t$ is a probability measure on $M$.
\end{proof}

Thus far, we have proven the existence of a measure $\mu_{t_0}$ for $t_0 \in [0,h)$ such that for $t \equiv 0$ (mod $h$), 
$$\mu_{t_0} (A) = \int_{x_0 \in M, s \in S} P(\varphi(t,x_0,s)\in A) d\mu_{t_0} .$$

\noindent The following theorem relates the collection of invariant measures $\{ \tilde{\mu}_{t_0} \}$ using the flow $\varphi$. This is the main result of the paper.

\begin{thm}\label{relatedmeasures}
Given invariant measure $\mu_0$, the measure $\mu_{t}$ defined by 
$$\mu_t(A) = \displaystyle \sum_{s \in S}{\int_{x_0 \in M}{P(\varphi(t,x_0,s)\in A)d\mu_0}}$$
is also invariant in the sense that $\mu_{t} = \mu_{t+nh}$ for $n \in \mathbb{N}$.
\end{thm}

\begin{proof}
We will show that $\mu_t = \mu_{t+h}$.  By induction, this implies that $\mu_t = \mu_{t+nh}$ for all $n \in \mathbb{N}$.  We have
$$\mu_{t+h}(A) = \sum_{s \in S}{\int_{x_0 \in M}{P(\varphi(t+h,x_0,s)\in A)d\mu_0}} .$$
Applying the definition of conditional probability,
$$\sum_{s \in S}{\int_{x_0 \in M}{P(\varphi(t+h,x_0,s) \in A)d\mu_0}} =  $$
$$\sum_{r \in S}{\int_{y \in M}{\left [P(\varphi(t,y,r) \in A) \sum_{s \in S}{\int_{x_0 \in M}{P(\varphi(h,x_0,s)\in dy \times \{r\})d\mu_0}}\right ]}} .$$
Loosely speaking, the probability that a trajectory beginning at $(x,s)$ will end in a set $A$ after a time $t+h$ is the product of the probability that a trajectory beginning at $(y,r)$ will end in $A$ after a time $t$ multiplied by the probability that a trajectory beginning at $(x,s)$ will end at $(y,r)$ after a time $h$, integrating over all possible pairs $(y,r)$.  Here, we have implicitly used the fact that the hybrid system is a Markov process to ensure that the state of the system at time $t+h$ given the state at time $h$ is independent of the initial state, and we have avoided the problem of time-inhomogeneity by considering trajectories that only begin at times congruent to $0$ (mod $h$).

Furthermore, we have that
$$\mu_h(dy \times \{r\}) =  \sum_{s \in S}{\int_{x_0 \in M}{P(\varphi(h,x_0,s)\in dy \times \{r\})d\mu_0}}$$
and
$$\mu_h(dy \times \{r\}) = d\mu_h(y,r); $$
so,
$$\mu_{t+h} (A) =  \sum_{r \in S}{\int_{y \in M}{P(\varphi(t,y,r)\in A)d\mu_h}} .$$
Since $\mu_0$ is invariant by assumption, $\mu_0 = \mu_h$.  Therefore,
$$\mu_{t+h} (A) =  \sum_{r \in S}{\int_{y \in M}{P(\varphi(t,y,r)\in A)d\mu_0}} = \mu_{t} (A) .$$
\end{proof}


\section{Examples}

Some examples of hybrid systems can be found in \cite{fitzhughnagurno, climatemodel}.  Here, we will examine two simple cases to illustrate the theory developed above.  

\subsection{A 1-D Hybrid System}

We begin with a 1-dimensional linear dynamical system with a stochastic perturbation:
$$\dot{x} = -x + Z_t $$
where $Z_t \in \{-1,1\}$.  Both components of this system have a single, attractive equilibrium point: for $Z_t=1$, this is $x=1$, and for $Z_t=-1$, $x=-1$.  At timesteps of length $h=1$, $Z_t$ is perturbed by a Markov chain given by the transition matrix $Q$.  $Q$ is therefore a $2 \times 2$ matrix of nonnegative entries, 
$$ Q = \left ( \begin{array}{cc}
P_{1 \rightarrow 1} & P_{1 \rightarrow -1} \\
P_{-1 \rightarrow 1} & P_{-1 \rightarrow -1} \\
\end{array} \right ) ,$$
where $P_{i \rightarrow j}$ gives the probability of the equilibrium point transitioning from $i$ to $j$ at each integer time step.  Since the total probability measure must equal $1$, 
$$ \sum_{j}{P_{i \rightarrow j}} = 1, \, \, \, i, j \in \{1,-1\} .$$
Furthermore, to avoid the deterministic case, we take $P_{i \rightarrow j} \neq 0$ for all $i,j$.

\begin{prop}\label{C2}
The stochastic limit set $C(x_0) = [-1,1]$ for all $x_0\in\mathbb{R}$.
\end{prop}

\begin{proof} 
We begin by showing that $C(x)\subset[-1,1]$: that is, that every possible trajectory in our system will eventually enter and never leave $[-1,1]$, meaning that no it is only possible to have $t^*\rightarrow\infty$ such that $x^*=y$ for $y\in[-1,1]$.
First, consider $x_0\in[-1,1]$.  If we are in state $Z_t=1$, then the trajectory is attracted upwards and bounded above by $x=1$; in state $Z_t=-1$, the trajectory is attracted downwards and bounded below by $x=-1$.  In both cases, the trajectory cannot move above $1$ or below $-1$, and so will remain in $[-1,1]$ for all time.  

Now, consider $x_0\notin[-1,1]$.  If the trajectory ever enters $[-1,1]$, by similar argument as above, it will remain in that region for all time.  So, it remains to show that $\varphi(t,x_0,Z_0)\in[-1,1]$ for some $t\in\mathbb{R}$.  First, take $x_0>1$.  In either state, the trajectory will be attracted downward, and will eventually enter $[1,2]$ at time $t_2$.  Once there, at the first timestep in which $Z_t=-1$ it will cross $x=1$ and enter $[-1,1]$.  And since we have taken all entries of the transition probability matrix $Q$ to be nonzero, there almost surely exists a time $t_3>t_2$ for which the state is $Z_t=-1$; then, the trajectory will enter $[-1,1]$ and never leave.  By similar argument, any trajectory starting at $x_0<-1$ will enter and never leave $[-1,1]$.  Thus, $C(x)\subset[-1,1]$.

Now, we must show that $[-1,1]\in C(x)$: that is, that for every trajectory $\varphi(t,x_0,Z_0)$ and every point $y\in[-1,1]$ there is $t^*\rightarrow\infty$ such that $\varphi(t^*,x_0,Z_0)\rightarrow y$.
To do this, we really only need to show that given any point $x_0 \in [-1,1]$ and any transition matrix $Q$, there almost surely exists some time $t^*$ with $\varphi(t^*,x_0,Z_0) = x^*$.  If one such time $t^*$ is guaranteed to exist, then we can iterate the process for a solution beginning at $(t^*, x^*)$ to produce an infinite sequence of times.  To show that $t^*$ exists, we calculate a lower bound on the probability that $\varphi(t_n,x_0,Z_0) = x^*$.  

Without loss of generality, suppose that $x_0 > x^*$.  We have already shown that any solution will enter $[-1,1]$, so take sup$(x_0) = 1$.  From here, we can calculate the minimum number of necessary consecutive periods, $k$, for which $Z_n = -1$ in order for a solution with $x_0 = 1$ to decay to $x^*$.  The probability of this sequence of $k$ consecutive periods occurring is given by 
$$ P_{1}(k) = (P_{1 \rightarrow -1})(P_{-1 \rightarrow -1})^{k-1}$$
if $Z_0 = 1$ and 
$$P_{-1}(k) = (P_{-1 \rightarrow -1})^k$$
if $Z_0 = -1$.  Thus, for some $t^* \in [0, k]$,
$$P(\varphi(t^*,x_0,Z_0)=x^*) \geq \mbox{min}(P_{1}(k), P_{-1}(k)) > 0 $$ 
since $P_{i \rightarrow j} > 0$.  So,
$$ P(t^* \notin [0,k]) \leq 1 - P(x^*) < 1 $$  
and
$$ P(t^* \notin [0,mk]) \leq (1 - P(x^*))^m .$$  
As $m \rightarrow \infty$, $(1 - P(x^*))^m \rightarrow 0$.  So, with probability $1$, there exists $t^*$ with $\varphi(t^*,x_0,Z_0) = x^*$.

By similar argument, for $x_0<x^*$ and all $x^*\in(-1,1)$, we can find a time sequence $\{t_n\}$ such that $\varphi(t_n,x_0,Z_0) = x^*$.  So, we know that for all $x^*\in (-1,1)$, $x^*\in C(x)$.

So, we have proven that $[-1,1]\subseteq C(x)$ and $(-1, 1) \subseteq C(x)$.  Since $C(x)$ must by definition be closed, $C(x) = [-1,1]$.
\end{proof}

We can study the behavior of this system numerically.  Figure 1 (left) depicts a solution calculated for the transition matrix
$$ Q_1 = \left ( \begin{array}{cc}
.4 & .6 \\
.5 & .5 \\
\end{array} \right )$$
with initial values $x_0 = 2$, $Z_0 = 1$.

As expected, the trajectory enters the interval $(-1,1)$ and stays there for all time, oscillating between $x = -1$ and $x = 1$.  Intuitively, it seems that the trajectory will cross any $x^*$ in this interval repeatedly, so that indeed $C(x) = [-1,1]$.  This is not quite so clear for the transition matrix
$$ Q_2 = \left ( \begin{array}{cc}
.1 & .9 \\
.1 & .9 \\
\end{array} \right )$$
which yields the trajectory shown in Figure 1 (right) for $x_0 = 2$, $Z_0 = 1$.

\begin{center}
\includegraphics[trim={0mm 0mm 0mm 8mm},clip,width=50mm]{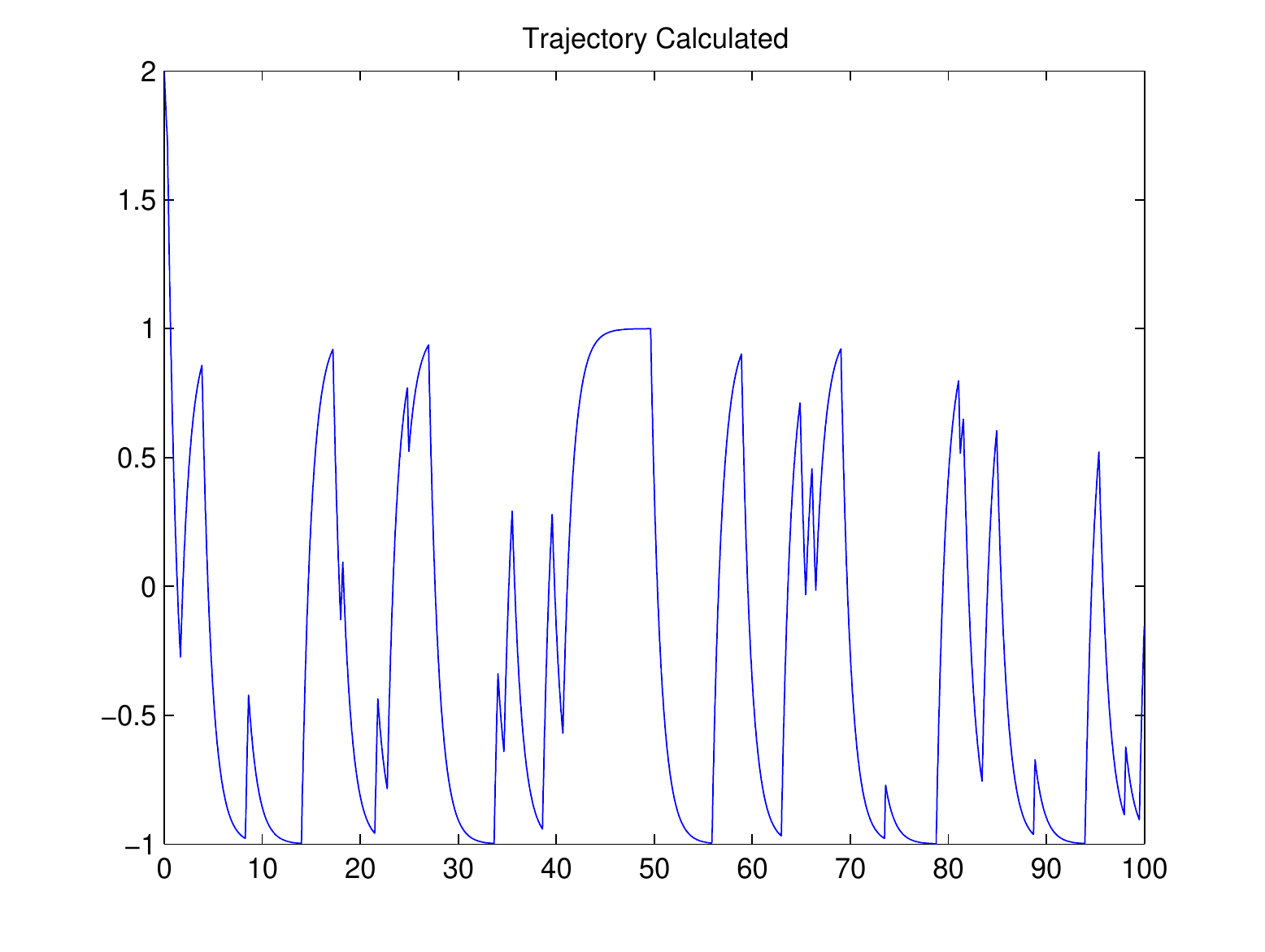}
\includegraphics[trim={0mm 0mm 0mm 8mm},clip,width=50mm]{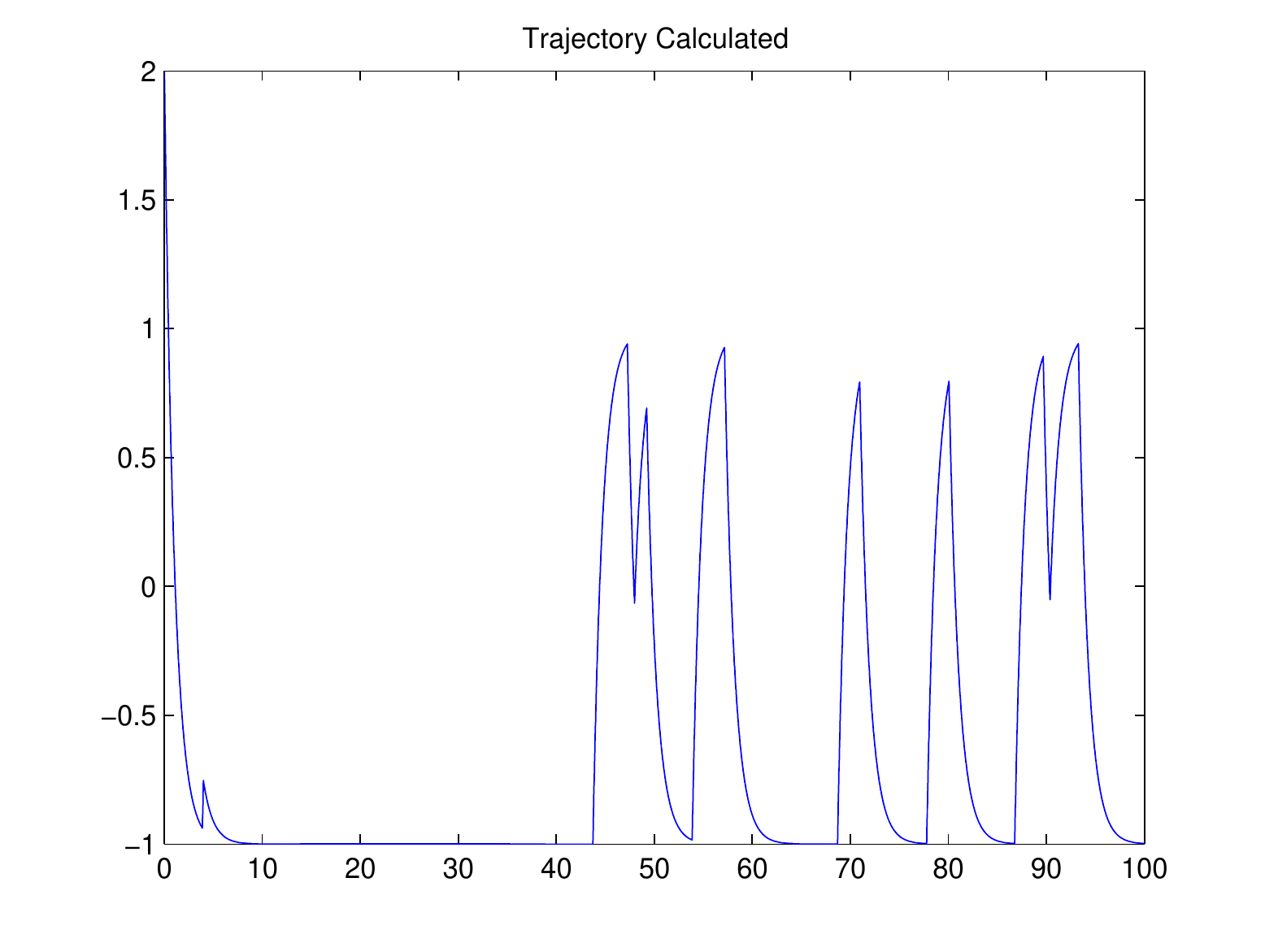}

Figure 1.  A sample trajectory for a hybrid system with transition matrix $Q_1$ (left) and $Q_2$ (right).
\end{center}

It may appear that some set of points near $x = 1$ might be crossed by our path only a finite number of times.  But, as proven above, any point in $(-1,1)$ will almost surely be reached infinitely many times as $t \rightarrow \infty$, so $C(x) = [-1,1]$.

Now, we consider the eigenvalues and eigenvectors of the transition matrices.  The eigenvector of $Q_1^T$ with eigenvalue $1$ is 
$$\vec{v} = \left ( \begin{array}{cc}
\frac{5}{11} \\
\frac{6}{11} \\
\end{array}\right ) ,$$
and the eigenvector of $Q_2^T$ with eigenvalue $1$ is 
$$\vec{v}\,' = \left ( \begin{array}{cc}
\frac{9}{10} \\
\frac{1}{10} \\
\end{array}\right ) .$$

These eigenvectors give the invariant measures on the state space $S$.  We know from Proposition \ref{inducedmeasure} that there also exists an invariant measure on $M$.  Here, since any trajectory in $M$ will almost surely enter $C(x) = [-1,1]$, the support of the invariant measure must be contained in $C(x)$.  It is not difficult to see that this invariant measure cannot be constant for all $t \in \mathbb{R}^+$.  Given any point $x_0 \in [-1,1]$, we know that at $t=1$, one of two things will have happened to the trajectory:

(i) it will have decayed exponentially toward $x=1$, if $Z_1 = 1$, or 

(ii) it will have decayed exponentially toward $x = -1$, if $Z_1 = -1$.

\noindent
In case (i), if a solution begins at $x_0=-1$ for $t=0$, then the solution will have decayed to a value of $1-(2 e^{-1})\approx 0.264$ by $t=1$ .  In case (ii) a solution beginning at $x_0 = 1$ for $t=0$ will decay to a value of $-1+2 e^{-1}\approx -0.264$.  Thus, if we are in case (i), all trajectories in $[-1,1]$ at $t = n$ will be located in $[0.264,1]$ at $t = n+1$.  If we are in case (ii), all will be in $[-1, -0.264]$.  It is not possible for any trajectory to be located in $[-0.264,0.264]$ at an integer time value.  But, clearly, some solutions will cross into this region, as depicted in Figure 2.  Therefore, no probability distribution will remain constant for all $t$ in the timeset $\mathbb{R}^+$.

\begin{center}
\includegraphics[width=100mm]{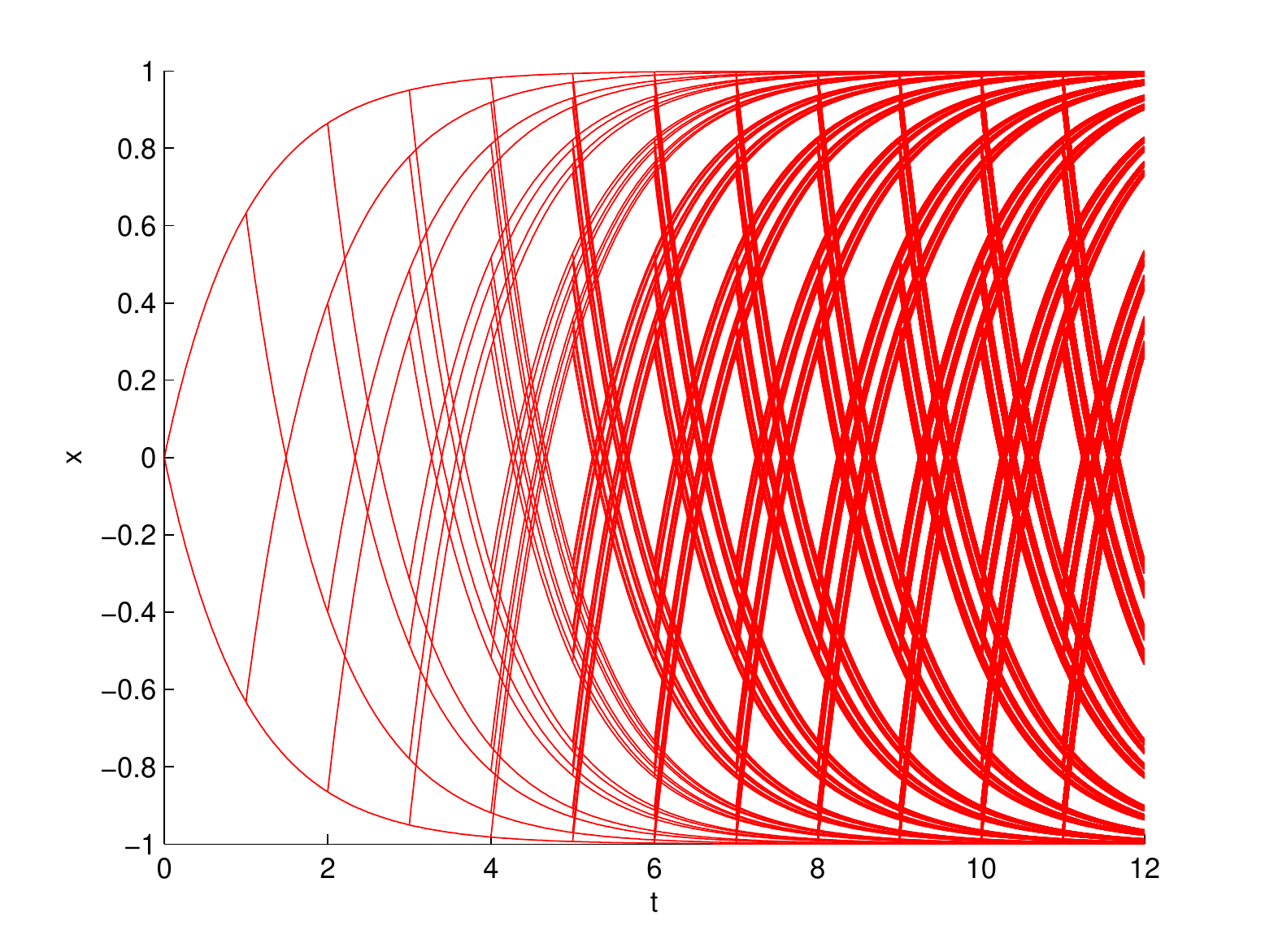}

Figure 2.  A spider plot showing all possible trajectories starting at $x_0=0$.
\end{center}

However, as Figure 2 suggests, there is some distribution that is invariant under $t \rightarrow t+n$ for $n \in \mathbb{N}$.  Approximations of the invariant measures at $t \in [0,1]$ for transition matrix $Q_1$ are shown in Figure 3. 

\begin{center}
\includegraphics[width=50mm]{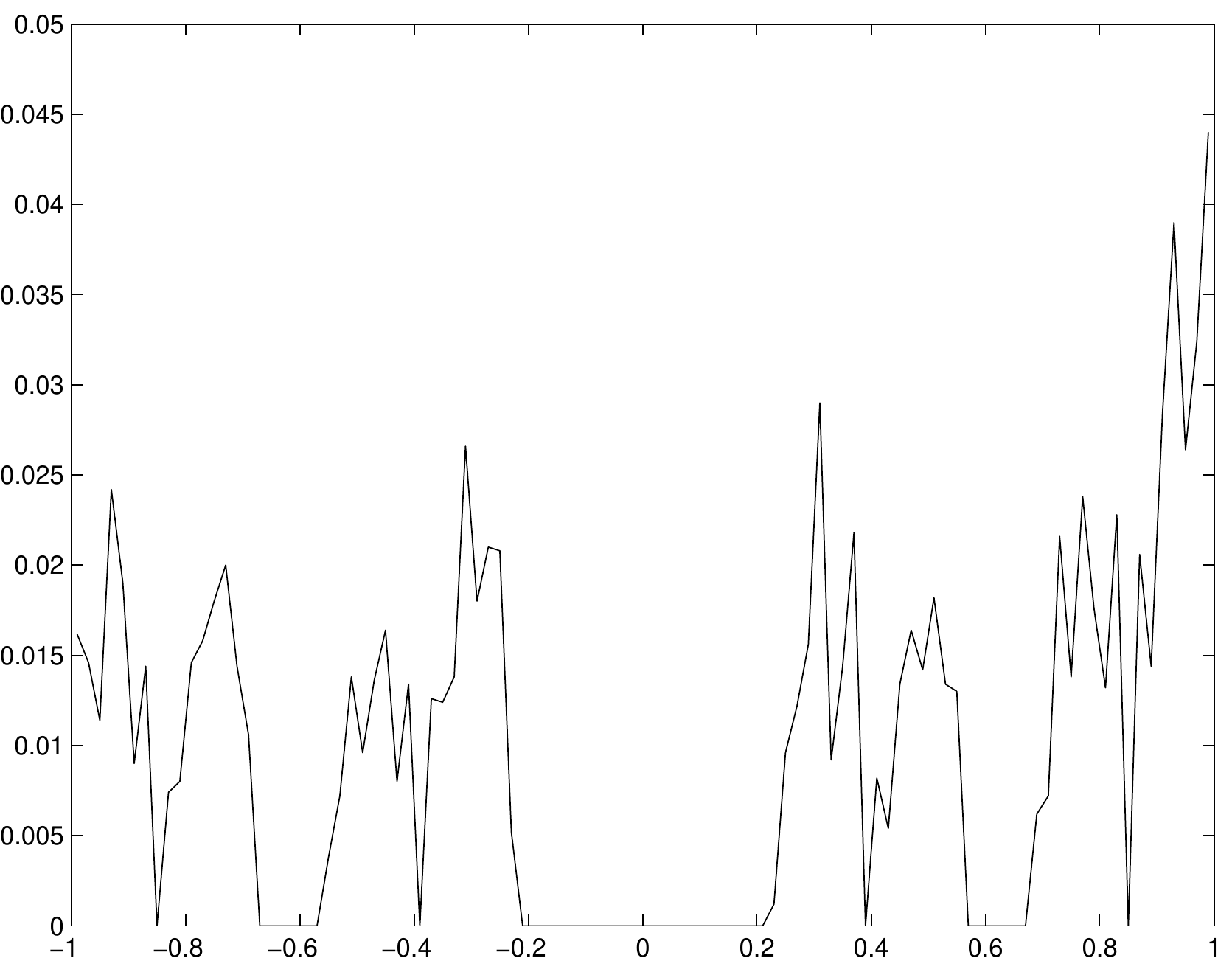}
\includegraphics[width=50mm]{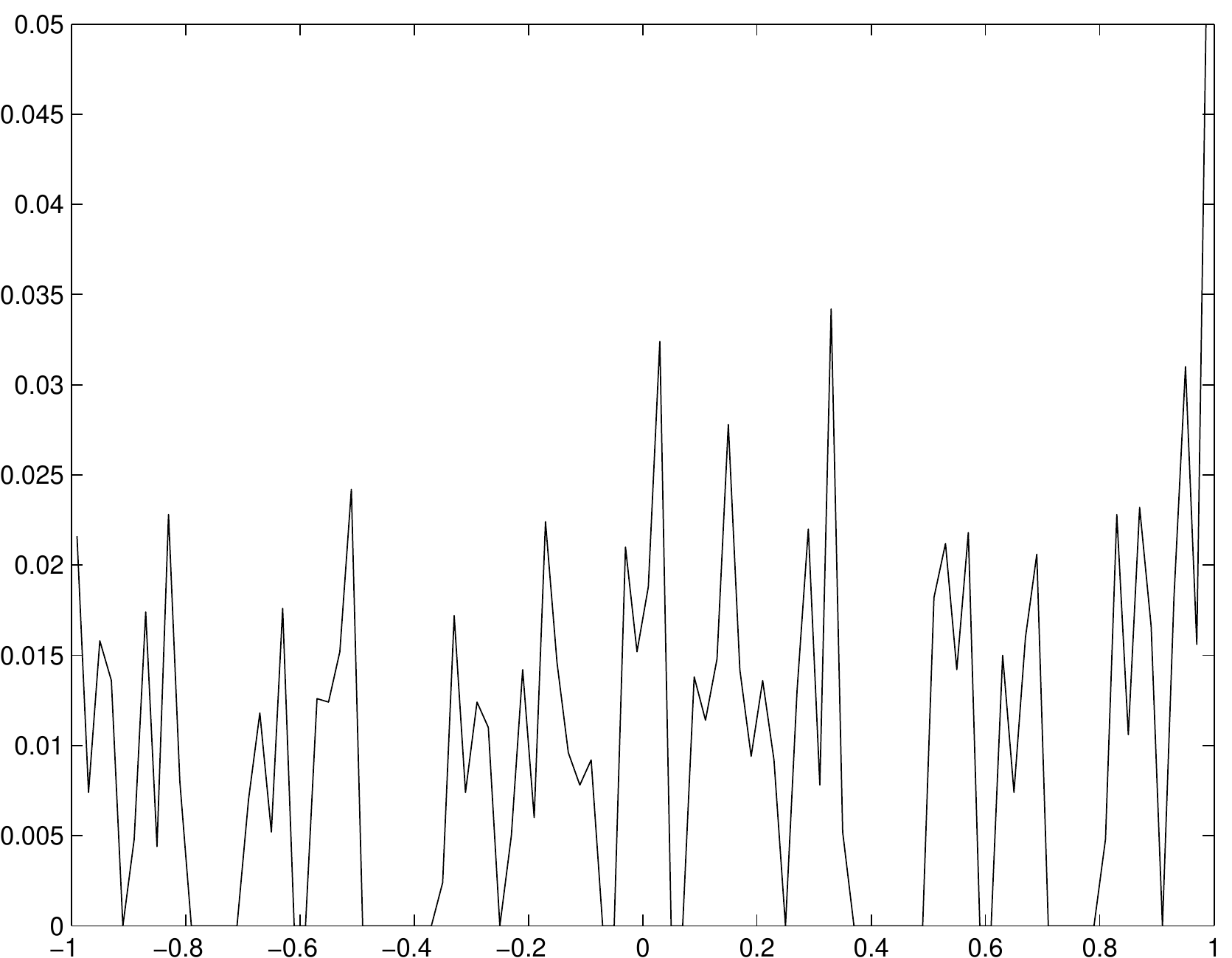} 
\includegraphics[width=50mm]{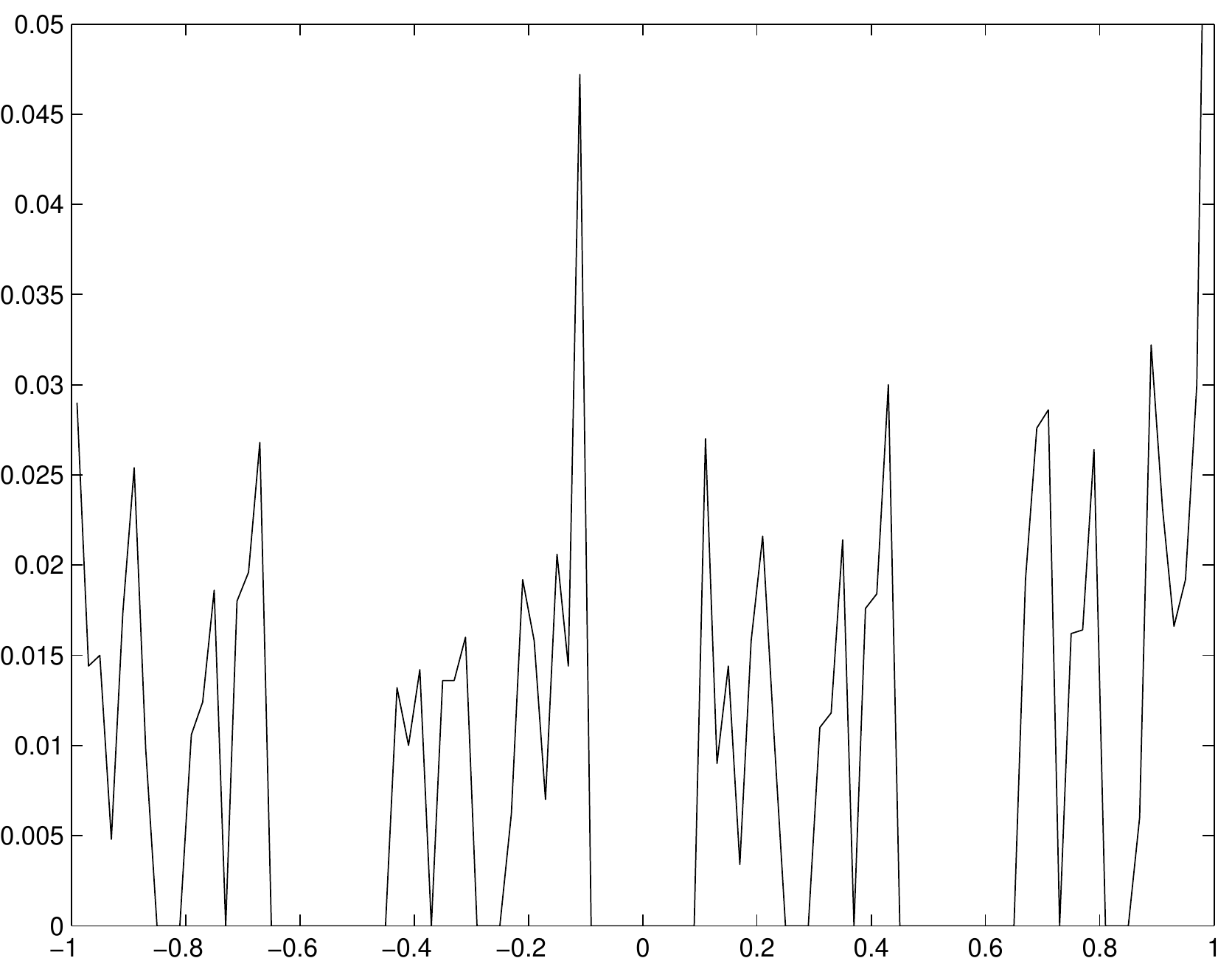} 
\includegraphics[width=50mm]{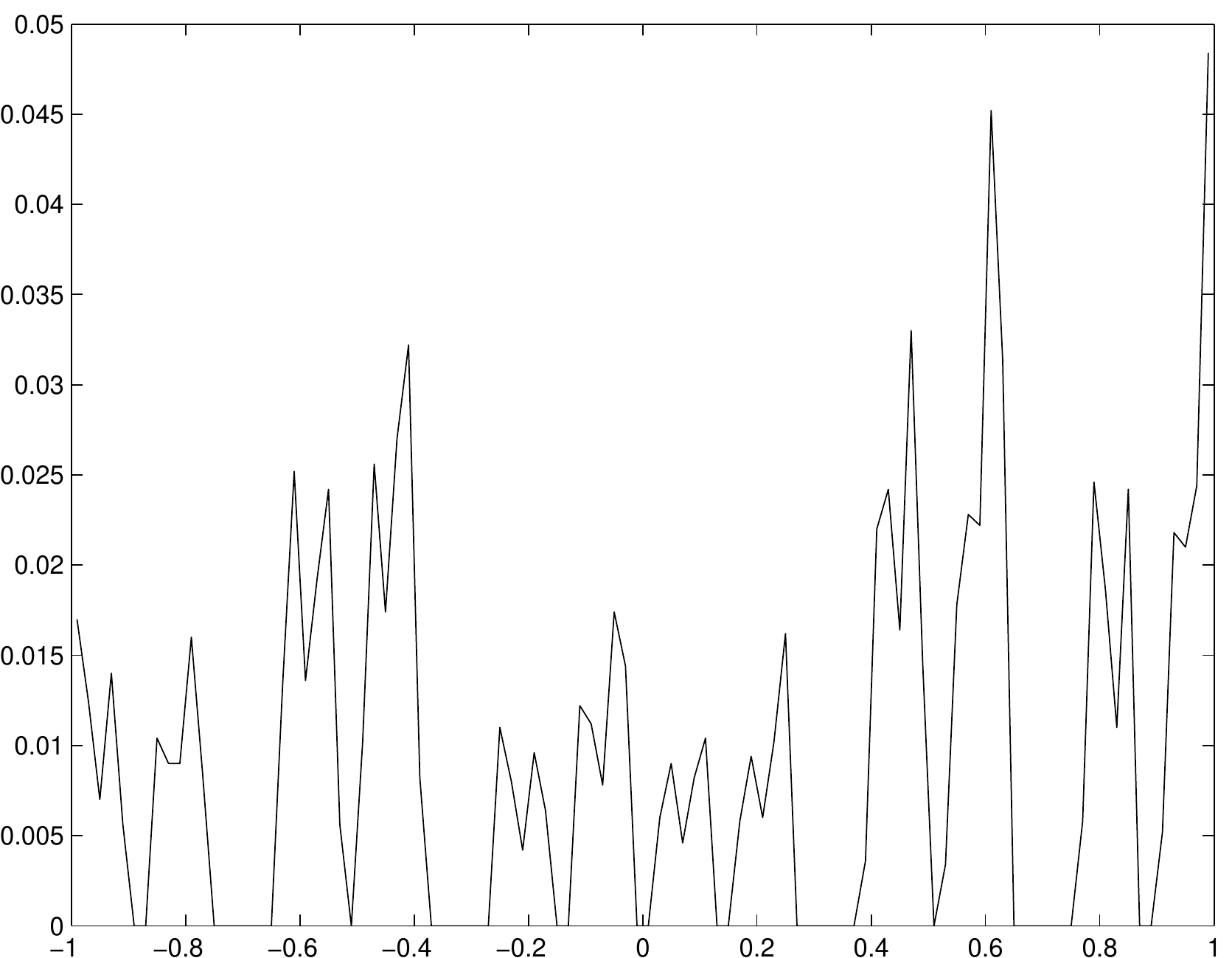} 
\includegraphics[width=50mm]{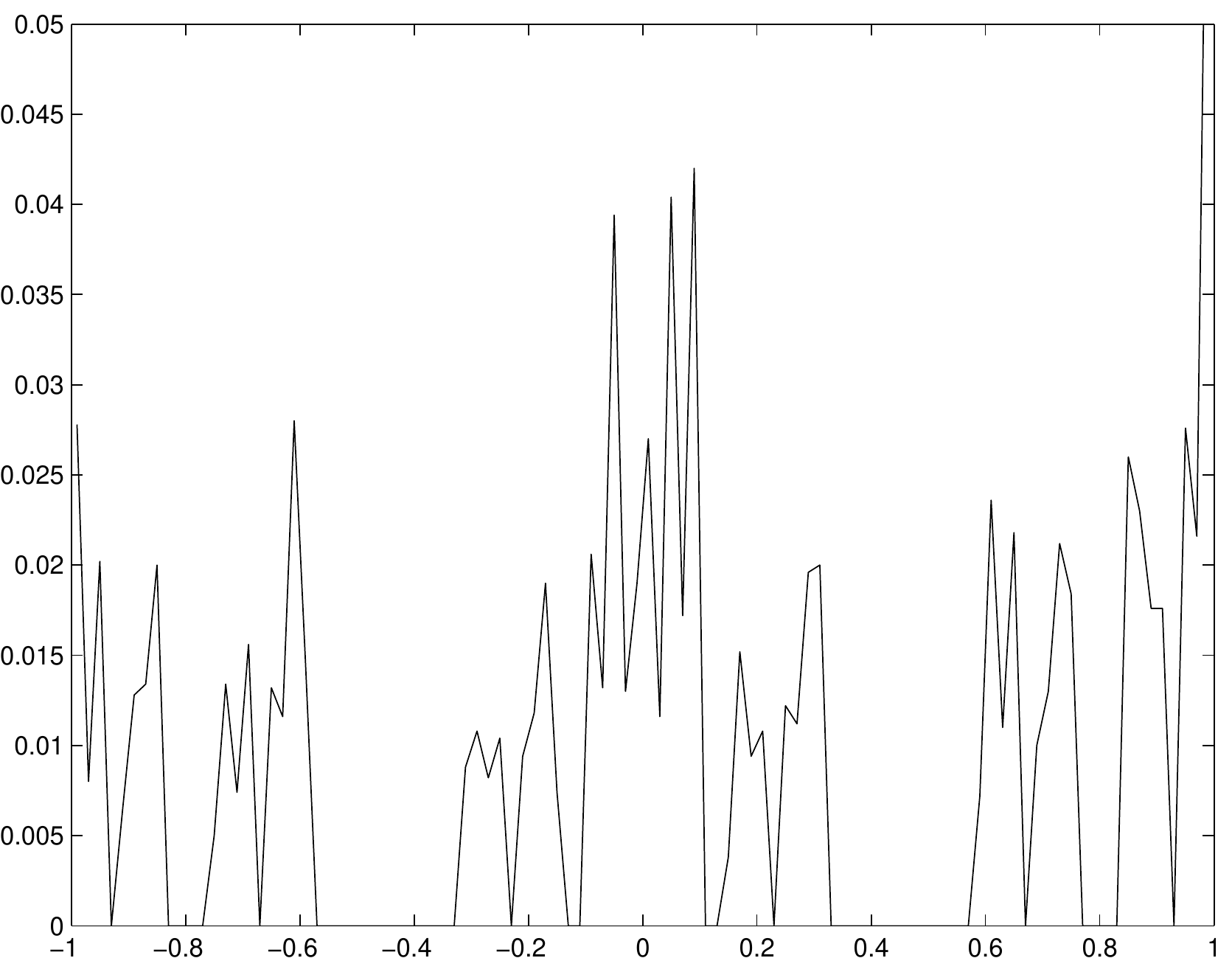} 
\includegraphics[width=50mm]{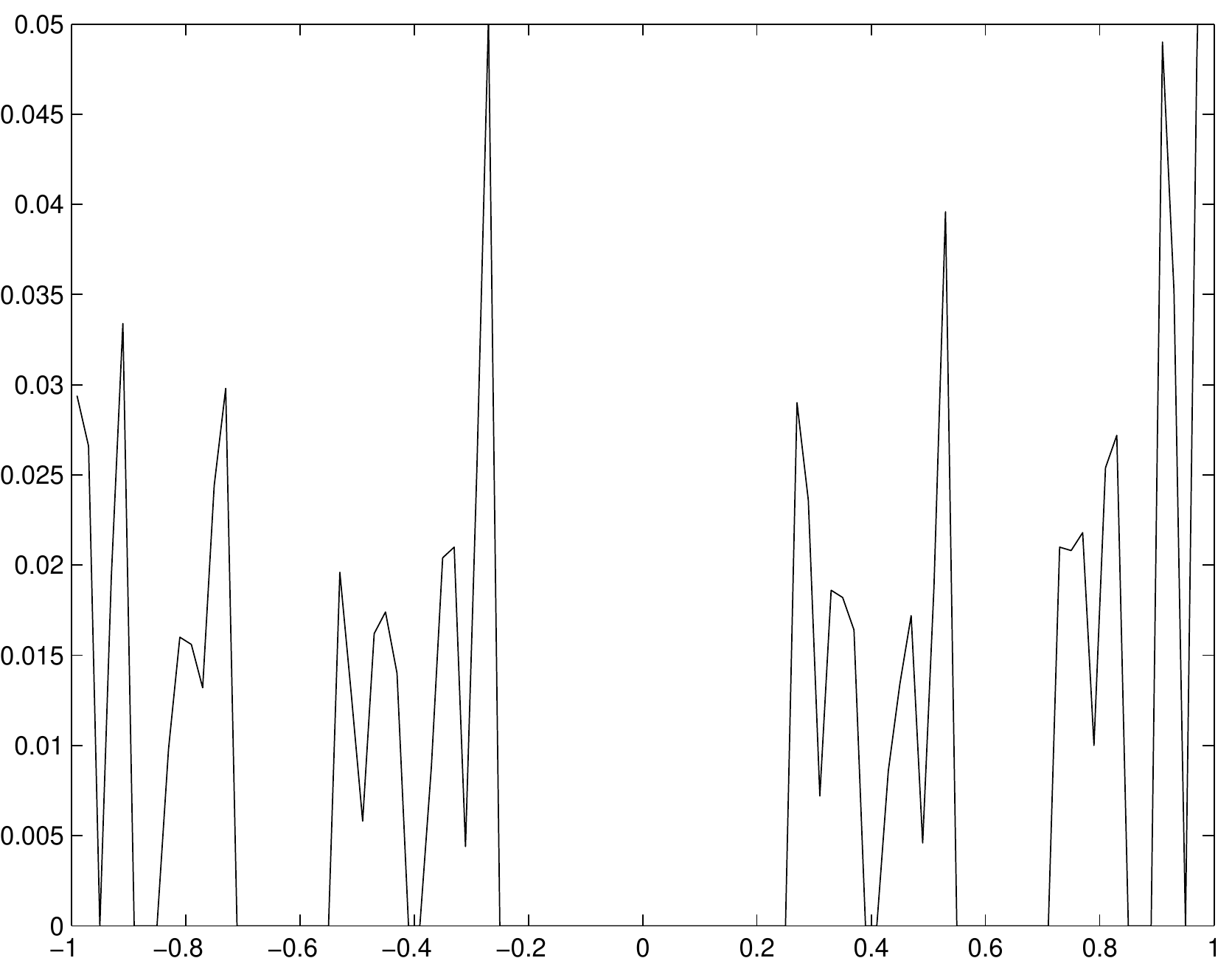} 

Figure 3.  The invariant measure $\tilde{\mu}_{0}$ for a hybrid system with transition matrix $Q_1$.
\end{center}


\subsection{A 2-D Hybrid System}

Our second example is a two-dimensional system used to model the kinetics of chemical reactors.  The general system $f(x_1,x_2)$ is given by

$$ \dot{x_1} = -\lambda x_1 - \beta (x_1 - x_c ) + BDa f(x_1,x_2)$$
$$ \dot{x_2} = -\lambda x_2+Daf(x_1,x_2) .$$
\cite{chemflow}

\noindent Here, we use the following simplified application of the system:

$$ \dot{x_1} = -x_1 - .15(x_1-1)+.35(1-x_2)e^{x_1}+Z_t(1-x_1) $$
$$ \dot{x_2} = -x_2 + .05(1-x_2)e^{x_1} .$$
This system is used to describe a Continuous Stirred Tank Reactor (CSTR). This type of reactor is used to control chemical reactions that require a continuous flow of reactants and products and are easy to control the temperature with. They are also useful for reactions that require working with two phases of chemicals.  

To understand the behavior of this system mathematically, we set our stochastic variable $Z_t=0$ and treat it as a deterministic system.  This system has three fixed points, approximately at $(.67,.09)$, $(2.64,.41)$, and $(5.90,.95)$; the former and latter are attractor points, while the middle is a saddle point, as shown in Figure 4.  The saddle point $(2.64,.41)$ creates a separatrix, a repelling equilibrium line between the two attracting fixed points.  These points, $(.67,.09)$ and $(5.90,.95)$, comprise the $\omega$-limit set of our state space.

\begin{center}
\includegraphics[trim={2cm 6cm 2.3cm 9.5cm},clip,width=100mm]{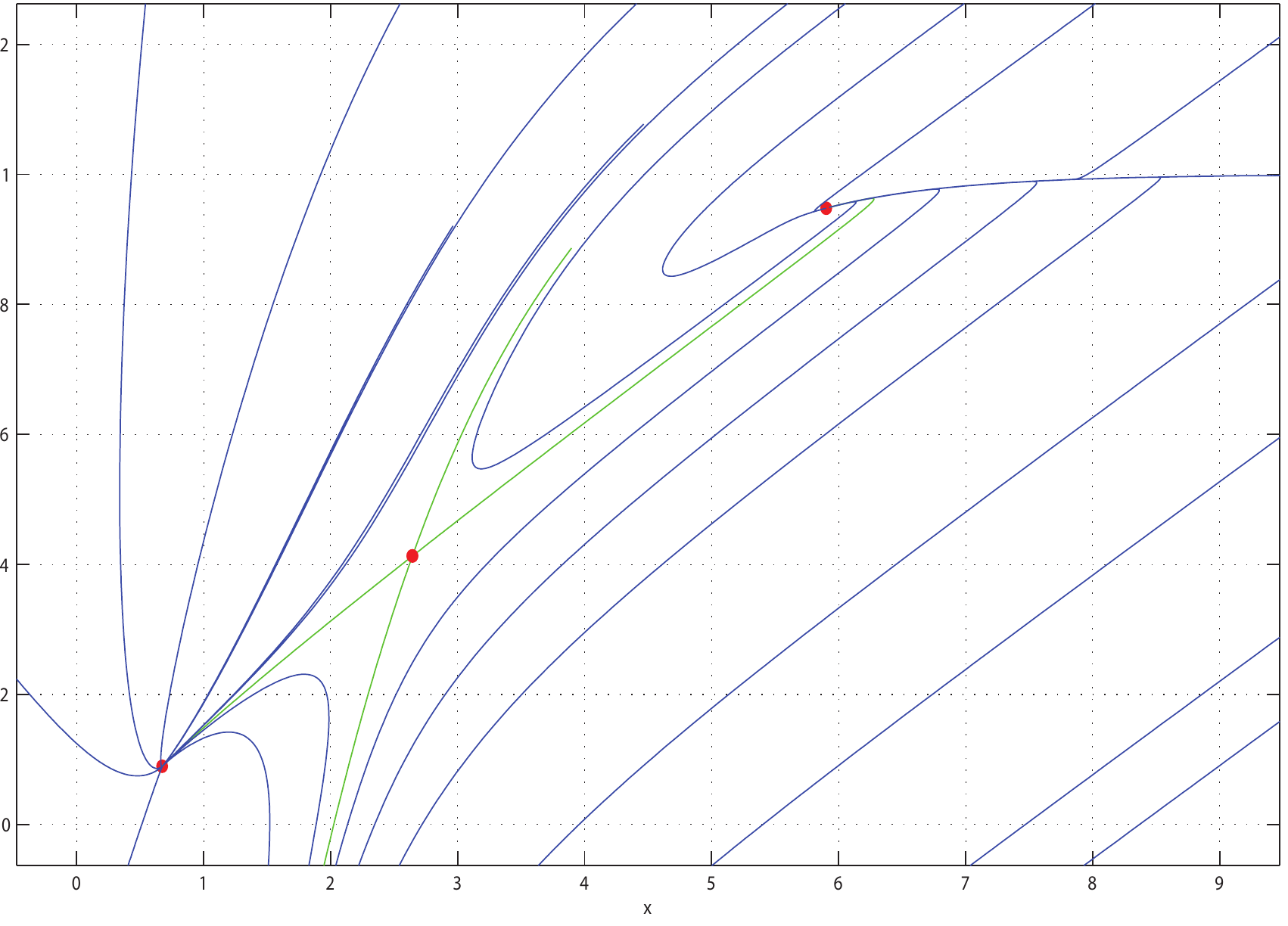}
\\Figure 4.  Phase plane of the deterministic system, $Z_n=0$.
\end{center}

With this information, we proceed to analyze the stochastic system.  As discussed above, the random variable here is $Z_t$, which in applications can take values between $-.15$ and $.15$.  To understand the full variability of this system, we take
$$Z_t \in \{-.15, 0, .15\} $$
with the transition matrix
$$\left ( \begin{array} {ccc}
.3&.3&.4 \\
.3&.3&.4 \\
.3&.3&.4 
\end{array}\right ) ,$$
yielding the phase plane in Figure 5.  We use red to indicate state 1 ($Z_t=-.15$), blue to indicate state 2 ($Z_t=0$), and green to indicate state 3 ($Z_t=.15$) for the corresponding fixed points, separatrices, and portions of trajectories.

\begin{center}
\includegraphics[width=100mm]{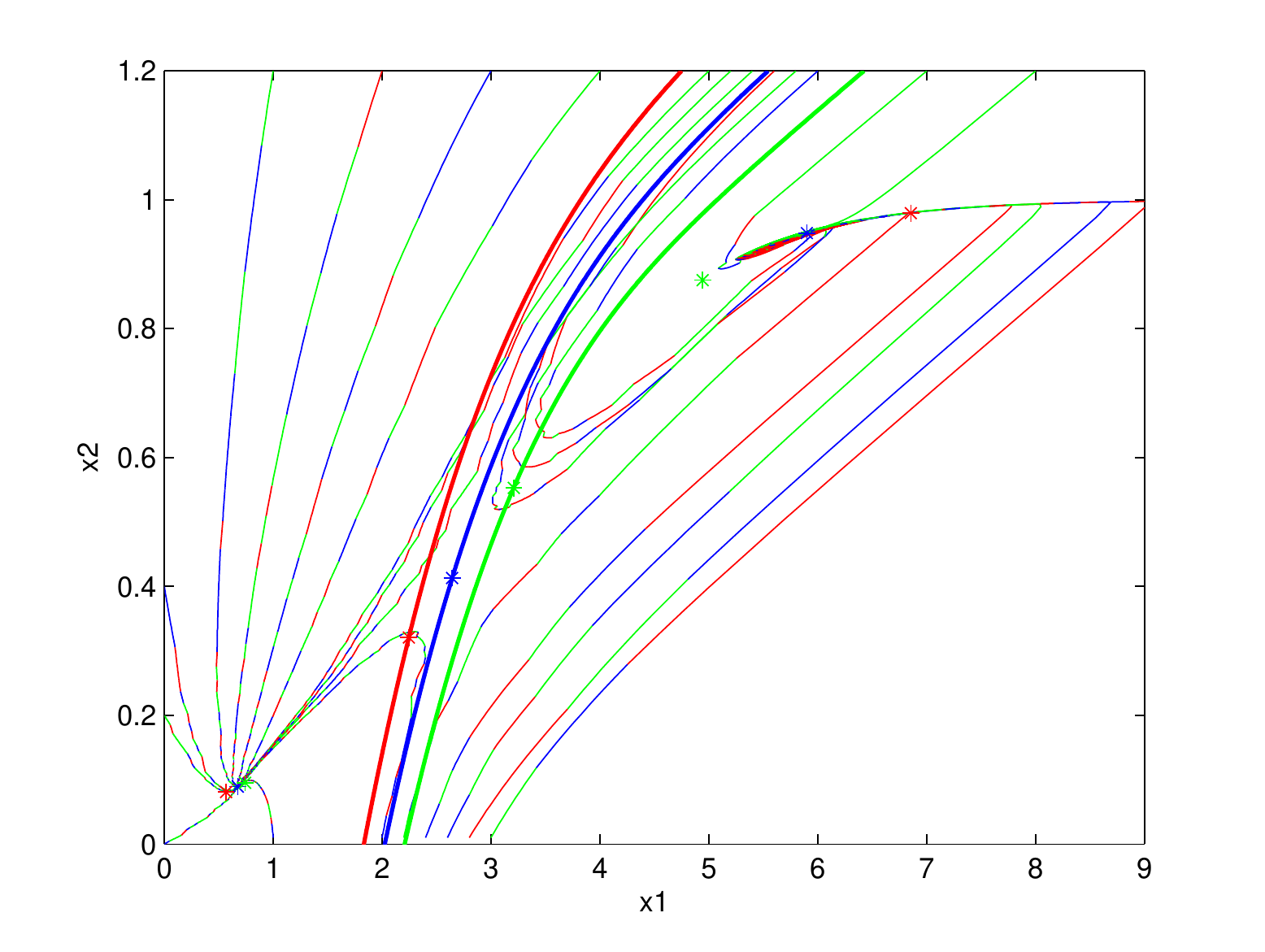}
\\ Figure 5.  Phase plane with randomness. 
\end{center}

We see that, for $x_0$ away from the separatrices, $\varphi(t,x_0,Z_0)$ behaves similarly to $\varphi(t,x_0)$.  Although state changes create some variability in a given trajectory, these paths move toward the groups of associated attracting fixed points, which define the stochastic limit sets for this system.  However, $\varphi(t,x_0,Z_0)$ for $x_0$ between the red and green separatrices is unpredictable; depending on the sequence of state changes for a given trajectory, it might move either to the right or the left of the region defined by the separatrices.  This area is the \textit{bistable region}, because a trajectory beginning within it has two separate stochastic limit sets.

For example, we have in Figure 6 a spider plot beginning in the bistable region at $(3.5,0.75)$.  A spider plot shows all possible trajectories starting from a single point in a hybrid system by, at each time step, taking every possible state.  Our previous coloring scheme still applies.

\begin{center} 
\includegraphics[width=100mm]{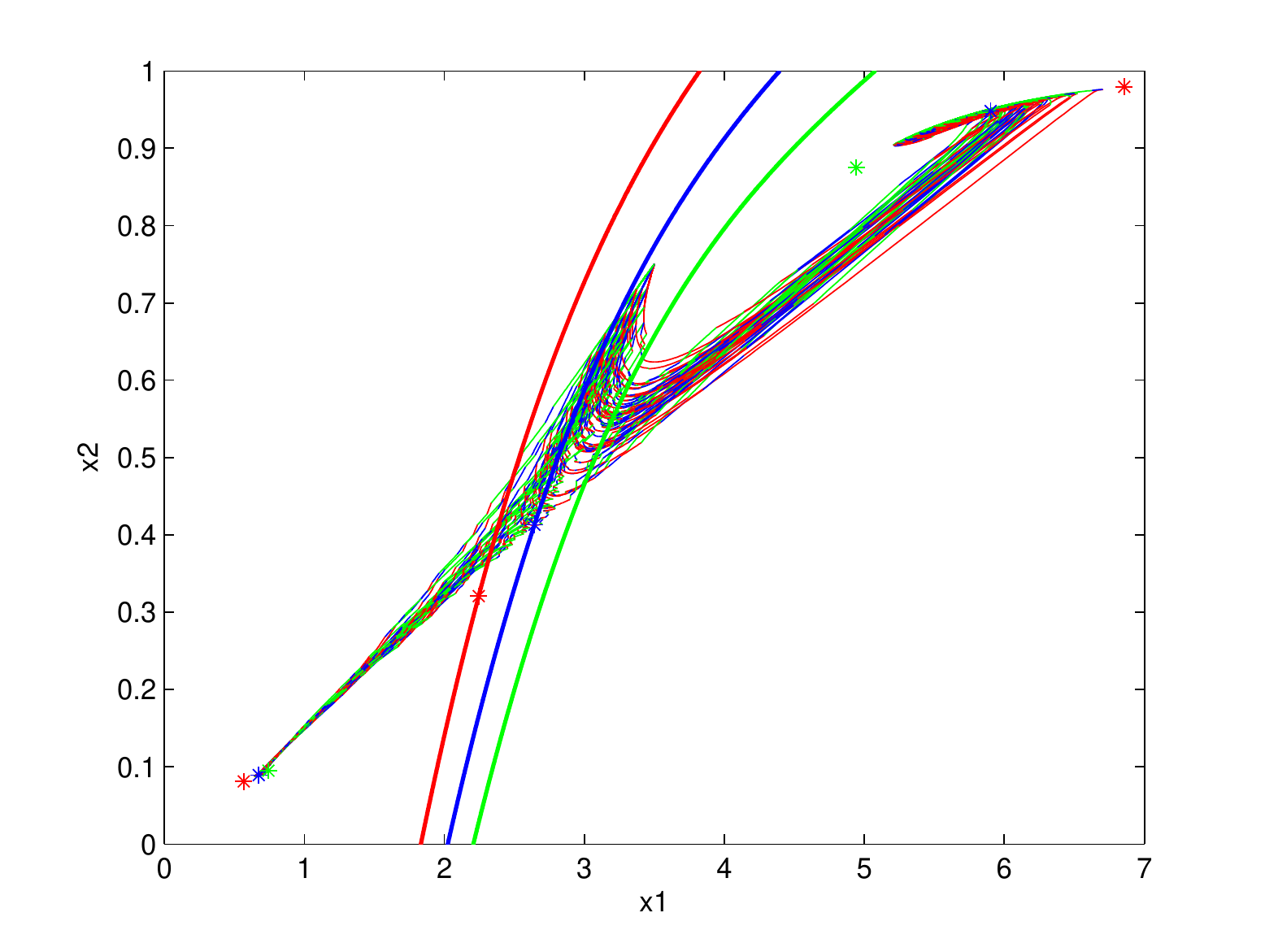}
\\ Figure 6.  Spider plot.
\end{center}

Thus, we see that the introduction of a stochastic element to a deterministic system can grossly affect the outcome of the system, as a trajectory can now cross any of the separatrices by being in a different state.  

The stochastic element also affects the behavior of the hybrid system around the invariant region.  In Figure 7, we show the path of a single trajectory in the invariant region defined by the fixed points near $(.67,0.9)$.  Plotting this trajectory for a long period of time approximates the invariant region that would appear if we ran a spider plot from the same point, but much more clearly.

\begin{center} 
\includegraphics[width=100mm]{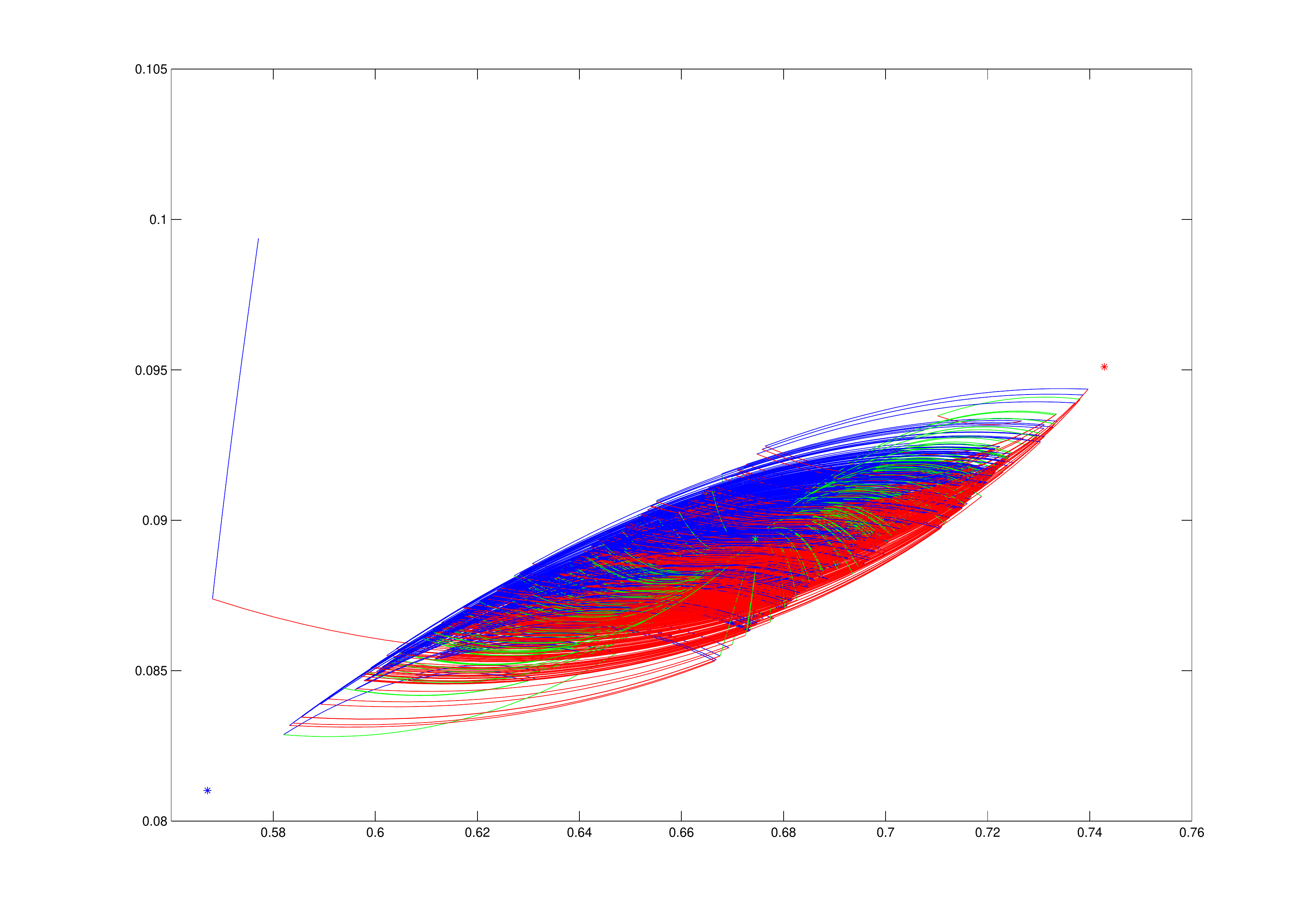}
\\ Figure 7.  Random trajectory.
\end{center}

As we saw in the $1$-dimensional system, considering the counts taken at specific times in the interval between two state changes, $h=1$ (since our state transitions occur on $\mathbb{N}$), yields a periodic set of invariant measures.  Similarly to Figure 3, Figure 8 shows the positions of our random trajectory in the invariant region at time $t$, mod $h$.

\begin{center} 
\includegraphics[width=50mm]{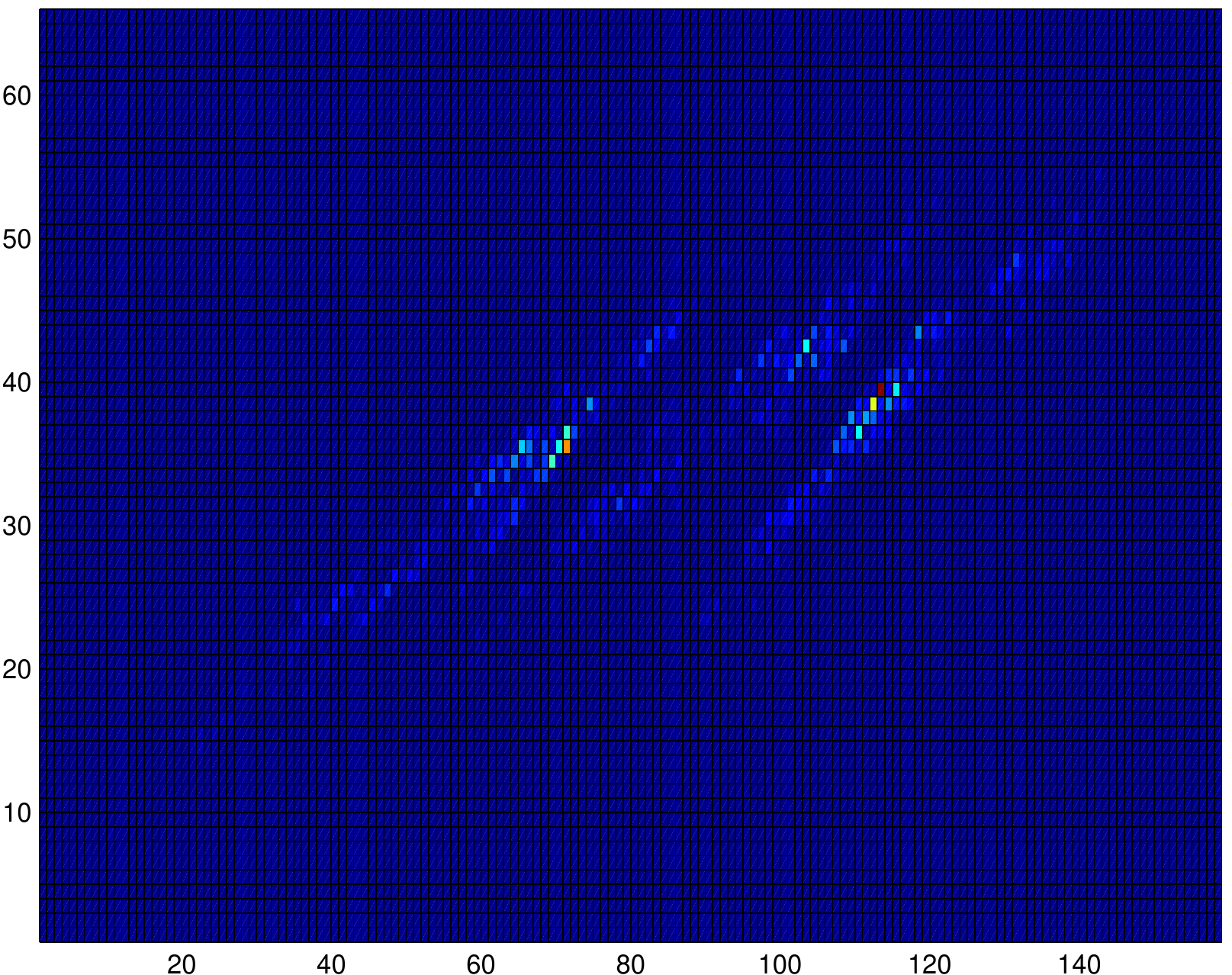}
\includegraphics[width=50mm]{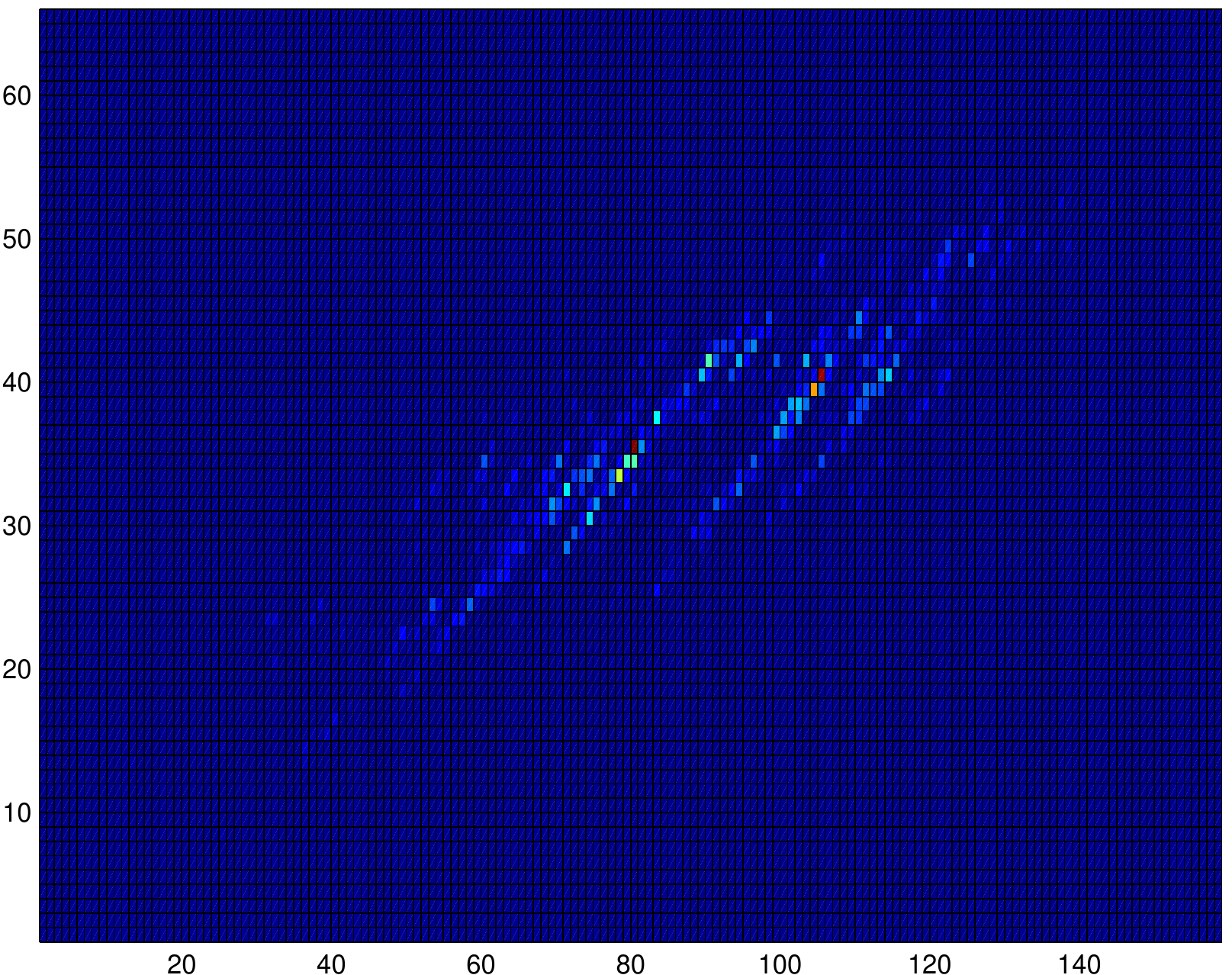}
\includegraphics[width=50mm]{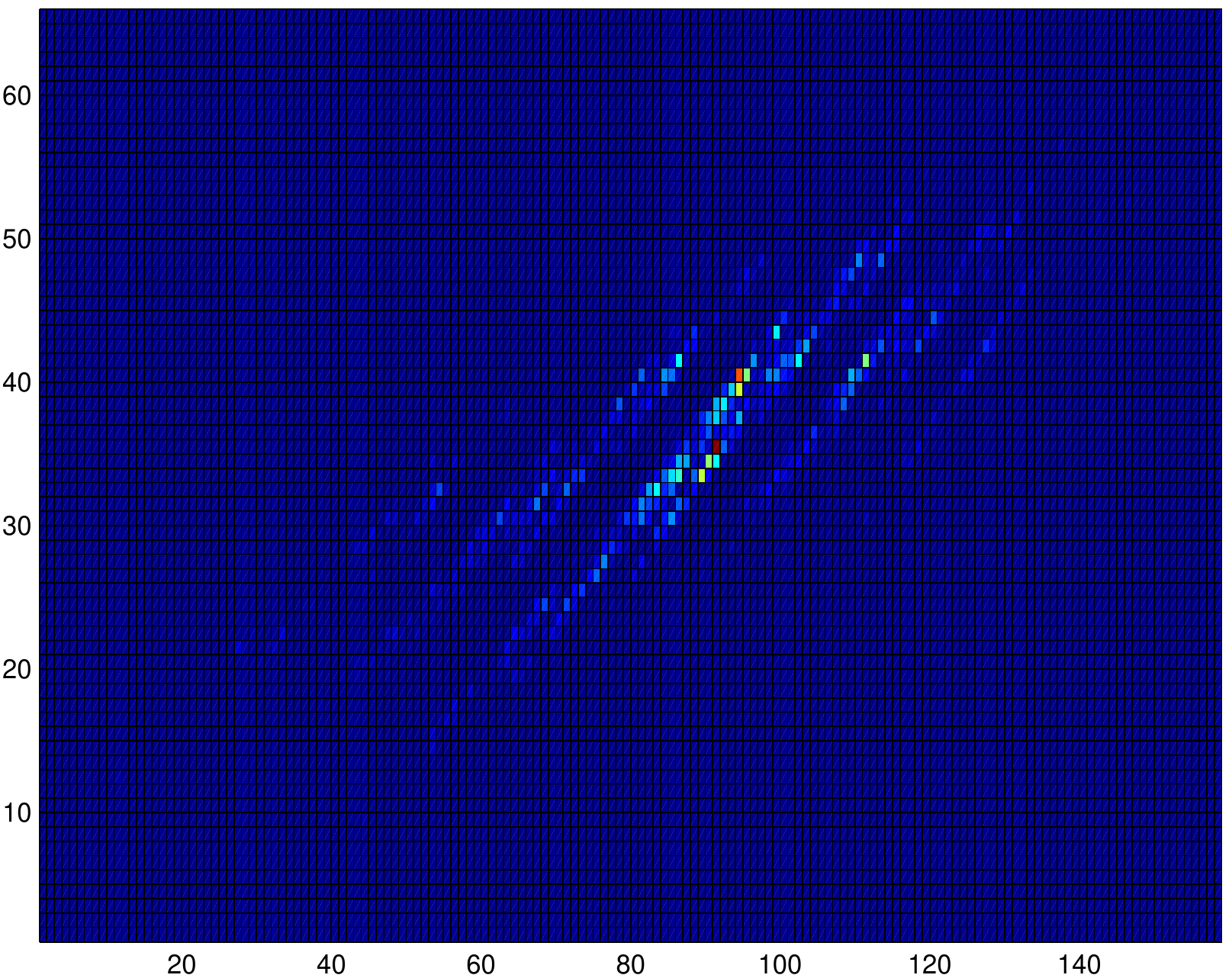}
\includegraphics[width=50mm]{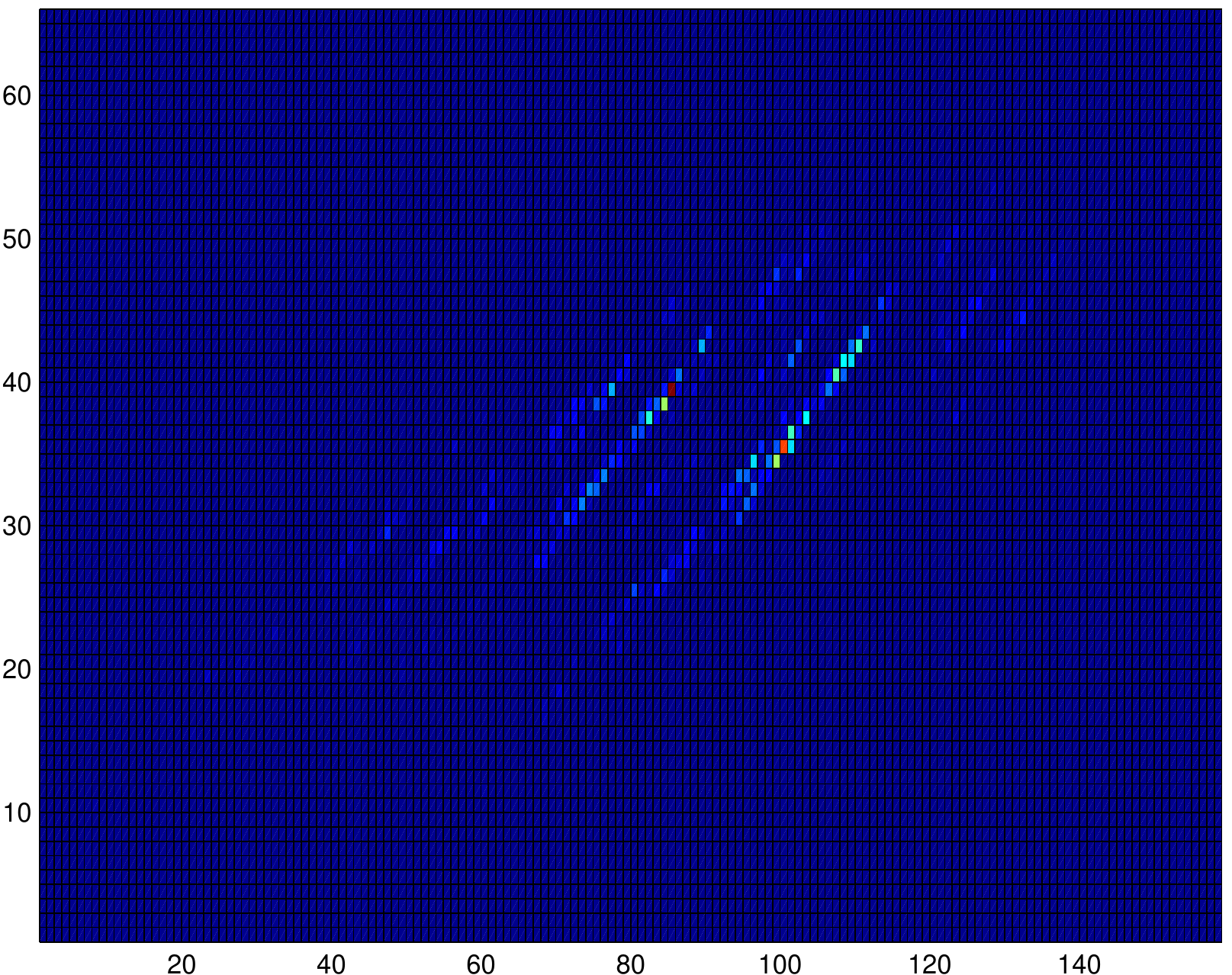}
\includegraphics[width=50mm]{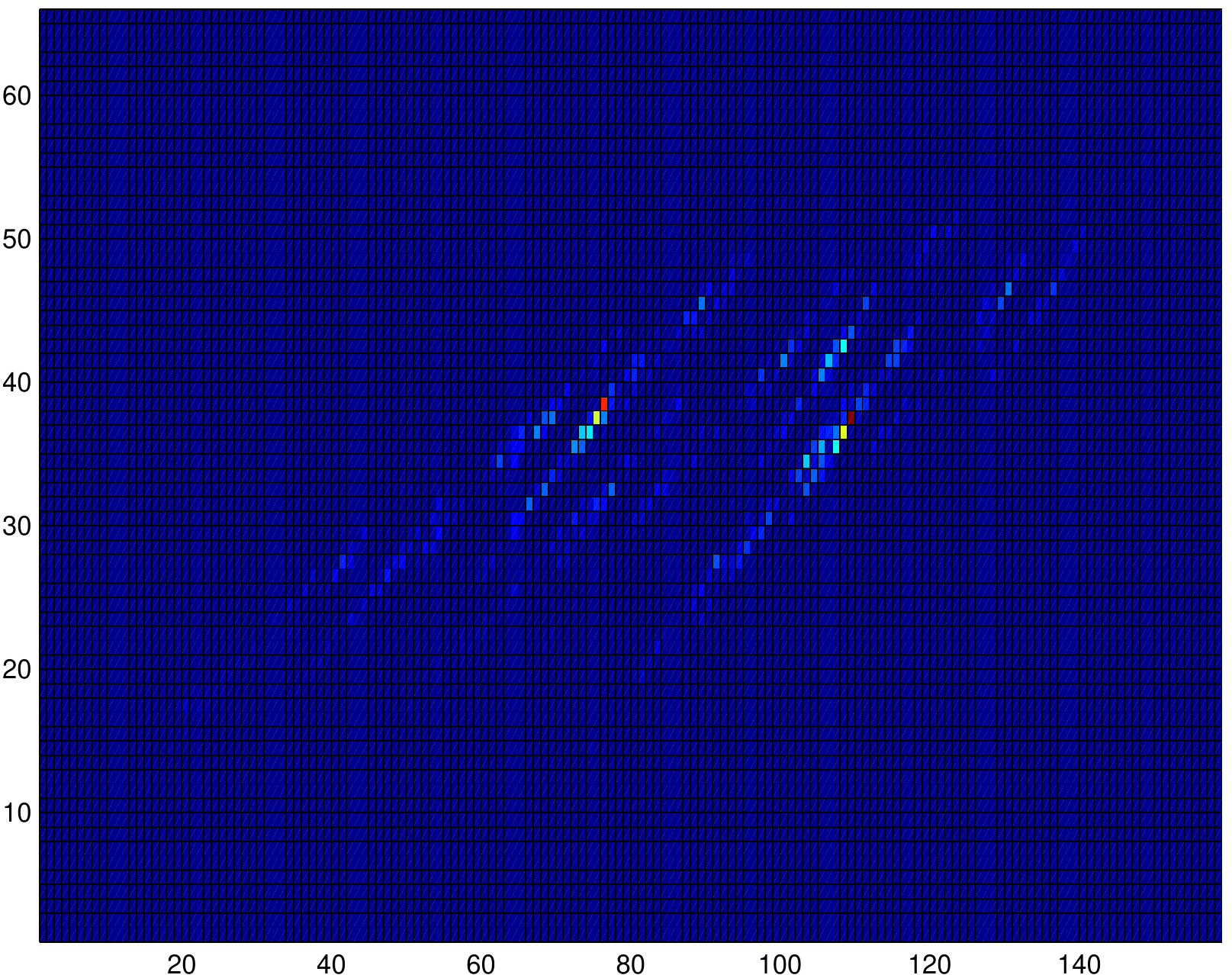}
\includegraphics[width=50mm]{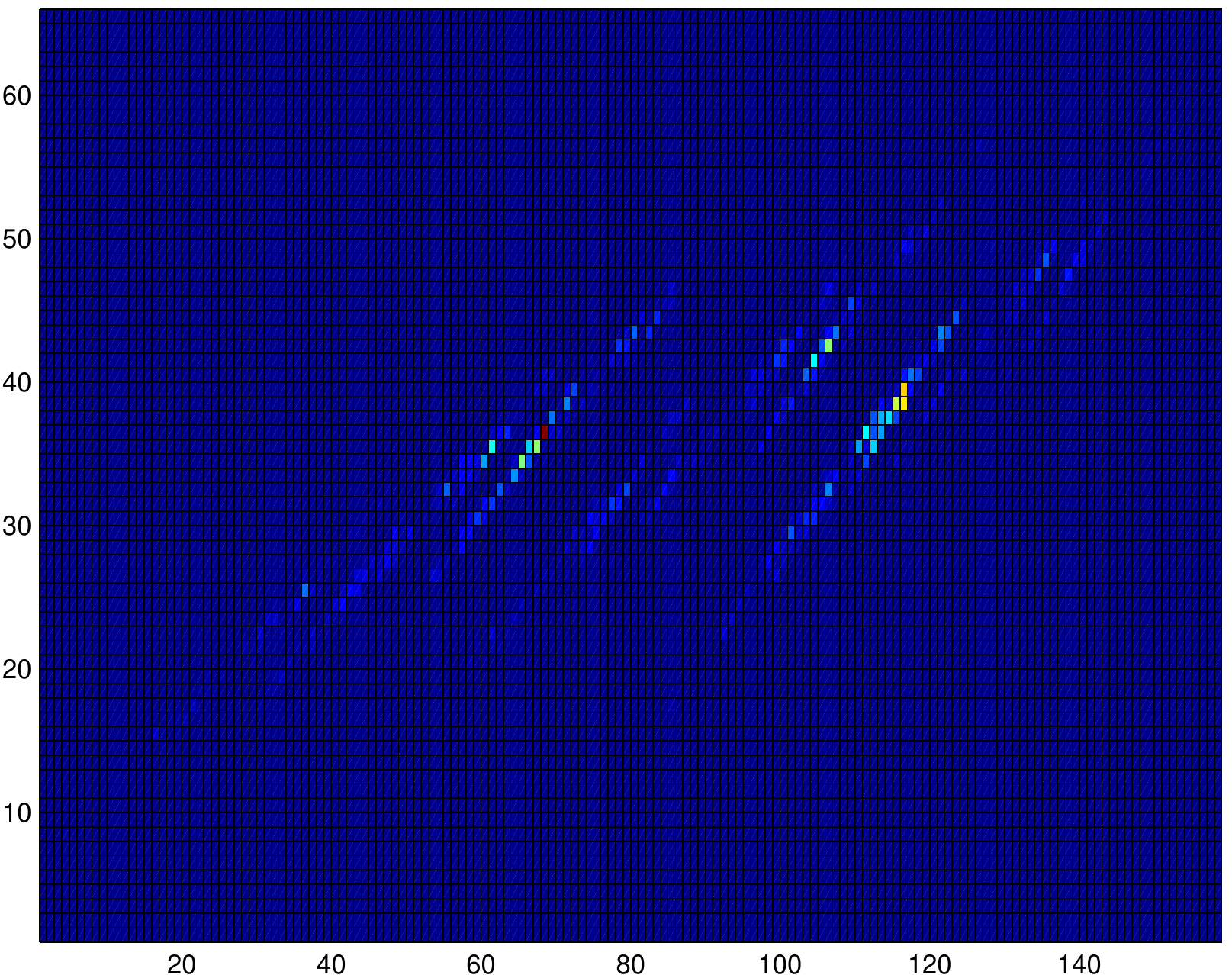}
\\ Figure 8.  Count of trajectory paths within one timestep.
\end{center}

A denser series of count images would show more clearly that the invariant measure at $t$ mod $h$ cycles continuously.


\section{Conclusion}

We have studied hybrid systems consisting of a finite set $S$ of dynamical systems over a compact space $M$ with a Markov chain on $S$ acting at discrete time intervals.  Such a hybrid system is a Markov process, which can be made time-homogeneous by discretizing the system.  Then, there exists a family of invariant measures on the product space $M \times S$, which can be projected onto a family of measures on $M$.  We have demonstrated a relation between the members of this family.

We have studied both a one-dimensional and a two-dimensional example of a hybrid system.  These examples provide insight into the stochastic equivalent of $\omega$-limit sets and yield graphical representations of the invariant measures on these sets.


\section{Acknowledgements}

We wish to recognize Kimberly Ayers for her helpful discussions and Professor Wolfgang Kliemann for his instruction and guidance.  We would like to thank the Department of Mathematics at Iowa State University for their hospitality during the completion of this work.  In addition, we would like to thank Iowa State University, Alliance, and the National Science Foundation for their support of this research.  Figure 4 was drawn using the `pplane8.m' MATLAB program.


\end{document}